\documentclass[leqno,12pt]{article}
\usepackage{latexsym}
\usepackage{amsmath}
\usepackage{amssymb}
\usepackage{enumerate}

\usepackage{theorem}
\newtheorem{theorem}{Theorem}[section]
\newtheorem{proposition}[theorem]{Proposition}
\newtheorem{lemma}[theorem]{Lemma}

\newtheorem{TCT}{Toponogov's comparison theorem for open triangles.}
\newtheorem{LTOT}{Lemma on thin open triangles.}

\theorembodyfont{\rmfamily}
\newtheorem{proof}{\textmd{\textit{Proof.}}}

\newtheorem{remark}[theorem]{Remark}
\newtheorem{example}[theorem]{Example}
\newtheorem{definition}[theorem]{Definition}

\newtheorem{acknowledgement}{\textmd{\textit{Acknowledgements.}}}

\makeatletter

\@addtoreset{equation}{section}
\makeatother

\newcommand{\qedd}{\hfill \Box}
\newcommand{\ve}{\varepsilon}
\newcommand{\lra}{\longrightarrow}
\newcommand{\wt}{\widetilde}
\newcommand{\wh}{\widehat}
\newcommand{\ol}{\overline}
\newcommand{\N}{\ensuremath{\mathbb{N}}}
\newcommand{\R}{\ensuremath{\mathbb{R}}}
\newcommand{\Sph}{\ensuremath{\mathbb{S}}}
\newcommand{\cA}{\ensuremath{\mathcal{A}}}

\newcommand{\cD}{\ensuremath{\mathcal{D}}}

\newcommand{\cH}{\ensuremath{\mathcal{H}}}
\newcommand{\cI}{\ensuremath{\mathcal{I}}}

\newcommand{\cP}{\ensuremath{\mathcal{P}}}
\newcommand{\cU}{\ensuremath{\mathcal{U}}}
\newcommand{\cR}{\ensuremath{\mathcal{R}}}

\def\inj{\mathop{\mathrm{inj}}\nolimits}

\def\Cut{\mathop{\mathrm{Cut}}\nolimits}
\def\Focal{\mathop{\mathrm{Foc}}\nolimits}
\def\FC{\mathop{\mathrm{FC}}\nolimits}

\title{Toponogov Comparison Theorem 
for Open Triangles
}

\author{Kei KONDO \ $\cdot$ \ Minoru TANAKA}
\date{}
\begin{document}
\maketitle

\vspace{-4.7mm}
\begin{center}
{\em \footnotesize in memory of the late professor Detlef Gromoll}
\end{center}
\vspace{3mm}

\begin{abstract}
The aim of our article is to generalize the Toponogov comparison theorem to 
a complete Riemannian manifold with smooth convex boundary. 
A geodesic triangle will be replaced by an open (geodesic) triangle standing on the boundary 
of the manifold, and a model surface will be replaced by the universal covering surface of 
a cylinder of revolution with totally geodesic boundary. 
\end{abstract}
\section{Introduction}\label{sec:int}

Cohn\,-Vossen is one of pioneers in global differential geometry. 
More than seventy years ago, 
he investigated the relationship between the total curvature 
and the Riemannian structure of complete open surfaces. 
He has given big influence to many geometers who research in global differential geometry, 
although he studied only $2$-dimensional manifolds in \cite{CV1} and \cite {CV2}. 
For example, Cohn\,-Vossen proved the following theorem known as a splitting theorem\,: 

\medskip

\begin{theorem}{\rm (\cite[Satz 5]{CV2})}
If a complete Riemannian $2$-manifold has non-negative Gaussian curvature and admits a straight line, then its universal covering space is isometric to Euclidean plane. 
\end{theorem}

\medskip\noindent
Toponogov (\cite{T2}) generalized this splitting theorem for any dimensional complete 
Riemannian manifolds with non-negative sectional curvature 
by making use of the Toponogov comparison theorem (\cite{T1}). 
It is well known that the Toponogov comparison theorem has produced many great classical results, e.g., the maximal diameter theorem by Toponogov (\cite{T1}), 
the structure theorem with positive sectional curvature by Gromoll and Meyer (\cite{GM}), 
and the soul theorem with non-negative sectional curvature by Cheeger and Gromoll (\cite{CG}). 
Besides the Toponogov comparison theorem, some techniques originating from 
Euclidean geometry also play a key role in the references above. 
The techniques such as drawing a circle or a geodesic polygon, 
and joining two points by a minimal geodesic segment are very powerful 
in the comparison geometry. 
Cohn\,-Vossen first introduced such techniques into global differential geometry (see 
\cite {CV1} and \cite{CV2}). 
The Toponogov comparison theorem enables us to make use of such a technique 
in the comparison geometry.

\bigskip

In 2003, Itokawa, Machigashira, and Shiohama generalized the Toponogov comparison theorem 
by means of the radial sectional curvature. 
Their result contains the original Toponogov comparison theorem 
as a corollary (see  \cite[Theorem 1.3]{IMS}). 
The model surface in the original Toponogov comparison theorem is 
a complete $2$-dimensional manifold of constant Gaussian curvature, 
but in \cite{IMS}, the model surface is replaced by a von Mangoldt surface of revolution. 
Here, a von Mangoldt surface of revolution is, by definition, 
a complete  surface of revolution homeomorphic to Euclidean plane 
whose radial curvature function is non-increasing on $[0, \infty)$. 
Very familiar surfaces such as paraboloids or $2$-sheeted hyperboloids are 
typical examples of a von Mangoldt surface of revolution. 
Hence, it is natural to employ a von Mangoldt surface of revolution as a model surface. 
The reason why a von Mangoldt surface of revolution is used 
as a model surface lies in the following property of the surface:

\medskip

\begin{theorem}{\rm (\cite[Main Theorem]{Tn})}
The cut locus of a point on a von Mangoldt surface of revolution is empty 
or a subray of the meridian opposite to the point.
\end{theorem}

\medskip\noindent
It would be {\bf impossible} to prove \cite[Theorem 1.3]{IMS} for general surfaces of revolution, 
because the cut locus of the surface appears as an obstruction, 
when we draw a geodesic triangle in the model surface. 
For example, the proof of \cite[Lemma 4.10]{KT2} suggests such an obstruction. 
In \cite{KT2}, the present authors very recently generalized \cite[Theorem 1.3]{IMS} 
for a surface of revolution  admitting a sector which has no pair of cut points.

\bigskip

Our purpose in this article is {\em to establish the Toponogov comparison theorem for Riemannian manifolds with convex boundary from the standpoint of the radial curvature geometry.}

\bigskip

Now we will introduce the radial curvature geometry for manifolds with boundary\,: 
We first introduce our model, which will be later employed 
as a reference surface of comparison theorems in complete Riemannian manifolds with boundaries. 
Let 
\[
\wt{M} := (\R, d\tilde{x}^2) \times_{m} (\R, d\tilde{y}^2)
\]
be a warped product of 
two $1$-dimensional Euclidean lines $(\R, d\tilde{x}^2)$ and $(\R, d\tilde{y}^2)$, 
where the warping function $m : \R \rightarrow (0, \infty)$ is a positive smooth function 
satisfying $m(0) = 1$ and $m'(0) = 0$. 
Then we call 
\[
\wt{X} := \left\{ \tilde{p} \in \wt{M} \ ; \ \tilde{x}(\tilde{p}) \ge 0 \right\} 
\]
a {\em model surface}. 
Since $m'(0) = 0$, the boundary 
\[
\partial \wt{X}:= \{ \tilde{p} \in \wt{X}\, ; \, \tilde{x}(\tilde{p}) = 0 \} 
\]
of $\wt{X}$ is {\em totally geodesic}. 
The metric $\tilde{g}$ of $\wt{X}$ is expressed as 
\begin{equation}\label{model-metric}
\tilde{g} = d\tilde{x}^2 + m(\tilde{x})^{2} d\tilde{y}^2
\end{equation}
on $[0, \infty) \times \R$. 
The function $G \circ \tilde{\mu} : [0,\infty) \rightarrow \R$ is called the 
{\em radial curvature function} of $\wt{X}$, 
where we denote by $G$ the Gaussian curvature of $\wt{X}$, 
and by $\tilde{\mu}$ any ray emanating perpendicularly from $\partial \wt{X}$ 
(notice that such $\tilde{\mu}$ will be called a $\partial \wt{X}$-ray). 
Remark that $m : [0, \infty) \rightarrow \R$ satisfies the differential equation 
\[
m''(t) + G (\tilde{\mu}(t)) m(t) = 0
\]
with initial conditions $m(0) = 1$ and $m'(0) = 0$.  
We define a sector
\[ 
\wt{X} (\theta) : = \tilde{y}^{-1} ((0, \theta))
\]
in $\wt{X}$ for each constant number $\theta >0$. 
Since a map $(\tilde{p}, \tilde{q}) \rightarrow (\tilde{p}, \tilde{q} + c)$, $c \in \R$, over $\wt{X}$ is an isometry, 
$\wt{X} (\theta)$ is isometric to $\tilde{y}^{-1} (c, c + \theta)$ for all $c \in \R$.\par
Hereafter, let $(X, \partial X)$ denote a complete Riemannian 
$n$-dimensional manifold $X$ with a smooth boundary $\partial X$. 
We say that $\partial X$ is {\em convex}, 
if all eigenvalues of the shape operator $A_{\xi}$ of $\partial X$ are 
non-negative in the inward vector $\xi$ normal to $\partial X$. 
Notice that our sign of $A_{\xi}$ differs from \cite{S}. 
That is, for each $p \in \partial X$ and $v \in T_{p}\partial X$, 
\[
A_{\xi}(v) = -\,(\nabla_{v} N)^{\top}
\]
holds. 
Here, we denote by $N$ a local extension of $\xi$, 
and by $\nabla$ the Riemannian connection on $X$.\par 
For a positive constant $l$, 
a unit speed geodesic segment $\mu : [0, l] \rightarrow X$ emanating from $\partial X$ 
is called a {\em $\partial X$-segment} if $d(\partial X, \mu(t)) = t$ on $[0, l]$. 
If $\mu : [0, l] \rightarrow X$ is a $\partial X$-segment for all $l > 0$, 
we call $\mu$ a {\em $\partial X$-ray}. 
Here, we denote by $d(\partial X, \, \cdot \, )$ the distance function to $\partial X$ 
induced from the Riemannian structure of $X$. 
Notice that a $\partial X$-segment 
is orthogonal to $\partial X$ by the first variation formula, and so a $\partial X$-ray is too.\par 
For any fixed two points $p, q \in X \setminus \partial X$, an {\em open triangle} 
\[
{\rm OT}(\partial X, p, q):=(\partial X, p, q \, ; \,\gamma, \mu_{1}, \mu_{2})
\]
in $X$ is defined by two $\partial X$-segments $\mu_{i} : [0, l_{i}] \rightarrow X$, $i = 1, 2$, 
a minimal geodesic segment $\gamma : [0, d(p, q)] \rightarrow X$, and $\partial X$ such that
\[
\mu_{1}(l_{1}) = \gamma (0) = p, \quad \mu_{2}(l_{2}) = \gamma (d(p, q)) = q.
\]
In this article, whenever an open triangle 
$
{\rm OT}(\partial X, p, q)=(\partial X, p, q\, ; \,\gamma, \mu_{1}, \mu_{2})
$ 
in $X$ is given, 
\begin{center}
$(\partial X, p, q\, ; \,\gamma, \mu_{1}, \mu_{2})$, as a symbol, 
\end{center}
always means that the minimal geodesic segment 
$\gamma$ is the side opposite to $\partial X$ emanating from $p$ to $q$, and that 
the $\partial X$-segments $\mu_{1}, \mu_{2}$ are sides emanating from $\partial X$ to $p$, $q$, 
respectively.\par 
$(X, \partial X)$ is said to have the 
{\em 
radial curvature $($with respect to $\partial X)$ 
bounded from below by that of $(\wt{X}, \partial \wt{X})$
} 
if, for every $\partial X$-segment $\mu : [0, l) \rightarrow X$, 
the sectional curvature $K_{X}$ of $X$ satisfies
\[
K_{X}(\sigma_{t}) \ge G (\tilde{\mu}(t))
\]
for all $t \in [0, l)$ and all $2$-dimensional linear spaces $\sigma_{t}$ spanned by $\mu'(t)$ 
and a tangent vector to $X$ at $\mu(t)$. 
For example, if the Riemannian metric of $\wt{X}$ is 
$d\tilde{x}^{2} + d\tilde{y}^{2}$, or $d\tilde{x}^2 + \cosh^{2} (\tilde{x})\,d\tilde{y}^{2}$, then 
$G (\tilde{\mu}(t)) = 0$, or $G (\tilde{\mu}(t)) = -1$, respectively. 
Furthermore, {\bf the radial curvature may change signs wildly} 
(e.g., \cite[Example 1.2]{KT1}, \cite{KT3}).

\bigskip

Our main theorem is now stated as follows\,: 

\begin{TCT}Let $(X, \partial X)$ be a complete connected Riemannian 
$n$-dimensional manifold $X$ with smooth convex boundary $\partial X$ 
whose radial curvature is bounded from below by that of 
a model surface $(\wt{X}, \partial \wt{X})$ with its metric $(\ref{model-metric})$. 
Assume that 
$\wt{X}$ admits a sector $\wt{X}(\theta_{0})$ which has no pair of cut points. 
Then, for every open triangle ${\rm OT}(\partial X, p, q) = (\partial X, p, q\,;\,\gamma, \mu_{1}, \mu_{2})$ 
in $X$ with 
\[
d (\mu_{1}(0), \mu_{2}(0)) < \theta_{0},
\]
there exists an open triangle 
$
{\rm OT}(\partial \wt{X}, \tilde{p}, \tilde{q}) 
= 
(\partial \wt{X}, \tilde{p}, \tilde{q}\,;\,\tilde{\gamma}, \tilde{\mu}_{1}, \tilde{\mu}_{2})
$ 
in $\wt{X}(\theta_{0})$ such that
\[
d(\partial \wt{X},\tilde{p}) = d(\partial X, p), \quad 
d(\tilde{p},\tilde{q}) = d(p, q), \quad 
d(\partial \wt{X},\tilde{q}) = d(\partial X, q)
\]
and that
\[
\angle\,p \ge \angle\,\tilde{p}, \quad  
\angle\,q \ge \angle\,\tilde{q}, \quad 
d (\mu_{1}(0), \mu_{2}(0)) \ge d (\tilde{\mu}_{1}(0), \tilde{\mu}_{2}(0)).
\]
Furthermore, if 
\[
d (\mu_{1}(0), \mu_{2}(0)) = d (\tilde{\mu}_{1}(0), \tilde{\mu}_{2}(0))
\] 
holds, then 
\[
\angle\,p = \angle\,\tilde{p}, \quad  \angle\,q = \angle\,\tilde{q}
\]
hold. 
Here $\angle\,p$ denotes the angle between two vectors $\gamma'(0)$ and 
$-\,\mu_{1}'(d(\partial X, p))$ in $T_{p}X$.
\end{TCT}

\bigskip\noindent 
Notice that we do not assume that $\partial X$ is connected in our main theorem. 
Moreover, remark that the opposite side 
$\gamma$ of ${\rm OT}(\partial X, p, q)$ 
does not meet $\partial X$ (Lemma \ref{lem3.10} in Section \ref{sec:Rem}). 
A related result for our main theorem is \cite[Theorem 3.4]{MS} of Mashiko and Shiohama. 
In \cite{MS}, they treat a pair $(M, N)$ of a complete connected Riemannian manifold $M$ 
and a compact connected totally geodesic hypersurface $N$ of $M$ such that 
the radial curvature with respect to $N$ is bounded from below by that of 
the model $((a, b) \times_{m} N, N)$, where $(a, b)$ denotes an interval, in their sense. 
Note that the radial curvature with respect to $N$ is bounded from below by that of our model 
$([0, \infty), d\tilde{x}^2) \times_{m} (\R, d\tilde{y}^2)$, 
if it is bounded from below by that of their model $((a, b) \times_{m} N, N)$. 
Thus, our Toponogov comparison theorem for open triangles is {\bf applicable to} the pair $(M, N)$.\par
We first prove the Toponogov comparison theorem for thin open triangles (see Definition \ref{def3.3} 
for the definition of thin open triangles). The first variation formula 
and some fundamental properties of the second variation formula will play key roles 
when we prove the Toponogov comparison theorem for thin open triangles. 
This new technique gives a new and sophisticated way of the proof of the original Toponogov 
comparison theorem. It was clarified in \cite{KT2} that the Toponogov comparison theorem holds 
for any geodesic triangles, if the Toponogov comparison theorem holds for thin geodesic triangles. 
This is also true for open triangles, if we imitate new techniques developed in \cite{KT2}. 

\bigskip

There are many examples of model surfaces satisfying the assumption in our main theorem. 
For example, it is clear that a model surface $(\wt{X}, \partial \wt{X})$ with the metric 
$d\tilde{x}^{2} + d\tilde{y}^{2}$, or $d\tilde{x}^2 + \cosh^{2} (\tilde{x})\,d\tilde{y}^{2}$ 
has no pair of cut points in a sector $\wt{X}(\theta)$ for each constant $\theta > 0$, respectively. 
Moreover, we have another example of model surfaces which have no pair of cut points in a sector:

\begin{example}\label{exa1.2}
Let  $\wt{M} := (\R, dt^{2}) \times_{m} (\Sph^{1}, d \theta^{2})$ be 
a warped product of a $1$-dimensional Euclidean line $(\R, dt^{2})$ 
and a unit circle $(\Sph^{1}, d \theta^{2})$ satisfying the next three conditions: 
\begin{enumerate}[{\rm ({C--}1)}]
\item
The warping function $m : \R \rightarrow (0, \infty)$ is a smooth even function satisfying 
$m(0) = 1$ and $m'(0) = 0$. 
\item
The radial curvature function $G (\tilde{\mu}(t)) = -m''(t)/ m(t)$ is non-increasing 
on $[0,\infty)$.
\item
$m'(t) \not= 0$ on $\R \setminus \{0\}$.
\end{enumerate}
Tamura (\cite{Tm}) proved that 
{\em the cut locus of a point $\tilde{p} \in \wt{M}$ with $\theta (\tilde{p}) = 0$ is 
the union of the meridian $\theta = \pi$ opposite to $\theta = 0$ and 
a subarc of the parallel $t = -t(\tilde{p})$}. 
Now, we introduce the Riemannian universal covering surface 
$
\wh{M} := \left(
\R \times_{m} \R, d\tilde{x}^2 + m(\tilde{x})^{2} d\tilde{y}^2
\right)
$ 
of $( \wt{M}, d t^{2} + m(t)^{2} d \theta^{2})$. 
It follows from Tamura's theorem above that the half space 
$
\wt{X} 
:= 
\left(
[0, \infty) \times_{m} \R, d\tilde{x}^2 + m(\tilde{x})^{2} d\tilde{y}^2
\right)
$ 
of $\wh{M}$ has no pair of cut points in a sector $\wt{X}(\theta)$ for each constant $\theta > 0$. 
For example, a model surface with the metric 
$d\tilde{x}^{2} + (e^{-\tilde{x}^2})^{2}d\tilde{y}^{2}$ is one of such models.
\end{example}

\medskip

We discuss applications of the Toponogov comparison theorem for open triangles in \cite{KT4}, which 
are splitting theorems of two types. Also, the Toponogov comparison theorem for open triangles 
in a weak form is discussed in the article. 

\bigskip

In the following sections, all geodesics will be normalized, unless otherwise stated.  

\begin{acknowledgement}
We are very grateful to Professor Ryosuke Ichida
for his helpful comments on the very first draft of this article. Finally, 
we would like to express to Professor Detlef Gromoll our deepest gratitude for 
his comment \cite{G} upon our work on radial curvature geometry (\cite{KT1}, \cite{KT2}). 
\end{acknowledgement}

\section{The sketch from Section 3 to Section 8}\label{sec:sch1}
Here, we sketch in the organization from Sections \ref{sec:FCL} to \ref{sec:TCTproof}, 
because we need many lemmas for proving our main theorem.

\bigskip

Throughout this section, let $(X, \partial X)$ be a complete connected Riemannian 
$n$-manifold $X$ with smooth {\bf convex} boundary $\partial X$ 
whose radial curvature is bounded from below by that of a model surface $(\wt{X}, \partial \wt{X})$. 

\bigskip

Our main purpose in Sections \ref{sec:FCL} to \ref{sec:LTOT} is to prove the following lemma, 
which is one of fundamental lemmas to establish the Toponogov comparison 
theorem for open triangles (Theorem \ref{thm4.9})\,:

\begin{LTOT}For every thin open triangle ${\rm OT}(\partial X, p, q)$ in $X$, 
there exists an open triangle 
${\rm OT}(\partial \wt{X}, \tilde{p}, \tilde{q})$ in $\wt{X}$ such that
\[
d(\partial \wt{X},\tilde{p}) = d(\partial X, p), \quad 
d(\tilde{p},\tilde{q}) = d(p, q), \quad 
d(\partial \wt{X},\tilde{q}) = d(\partial X, q)
\]
and that
\[
\angle\,p \ge \angle\,\tilde{p}, \quad  
\angle\,q \ge \angle\,\tilde{q}.
\]
\end{LTOT}

\bigskip

Thin open triangles are defined as follows:

\begin{definition}\label{def3.3}{\bf (Thin open triangle)}
An open triangle ${\rm OT}(\partial X, p, q)$ 
in $(X, \partial X)$ is called a {\em thin open triangle}, if 
\begin{enumerate}[{\rm ({TOT--}1)}]
\item
the opposite side $\gamma$ of ${\rm OT}(\partial X, p, q)$ to $\partial X$ 
emanating from $p$ to $q$ is contained in a normal convex neighborhood 
in $X \setminus \partial X$, and
\item
$L(\gamma) < \inj (\tilde{q}_{s})$ for all $s \in [0, d(p, q)]$,
\end{enumerate}
where $L(\gamma)$ denotes the length of $\gamma$, 
and $\tilde{q}_{s}$ denotes a point in $\wt{X}$ with 
\[
d(\partial \wt{X}, \tilde{q}_{s}) = d(\partial X, \gamma (s))
\]
for each $s \in [0, d(p, q)]$. 
\end{definition}

\bigskip\noindent
Here, the {\em injectivity radius} $\inj (\tilde{p})$ of a point $\tilde{p} \in \wt{X}$ 
is the supremum of $r > 0$ such that, for any point $\tilde{q} \in \wt{X}$ 
with $d(\tilde{p}, \tilde{q}) < r$, 
there exists a unique minimal geodesic segment joining $\tilde{p}$ to $\tilde{q}$. 
Remark that, for each point $\tilde{p} \in \wt{X} \setminus \partial \wt{X}$, 
the inequality 
$\inj(\tilde{p}) > d(\partial \wt{X}, \tilde{p})$ holds, 
if $\tilde{p}$ is sufficiently close to $\partial \wt{X}$.\par
Hence, Sections \ref{sec:FCL} and \ref{sec:LGV} are set up to prove 
Lemma on thin open triangles (Lemma \ref{lem3.8})\,: 
In Section \ref{sec:FCL}, we investigate the relationship between 
minimal geodesic segments in a complete connected Riemannian manifold $X$ 
with smooth boundary $\partial X$ and 
the focal cut locus of $\partial X$ (Lemma \ref{lem2.2}). 
Section \ref{sec:LGV} is the {\bf heart} of this article, i.e., 
we have Key Lemma (Lemma \ref{lem2.6}) of this article. 
Here Lemma \ref{lem2.6} is a comparison theorem of the Rauch type on length 
of $\partial X$-segments in variations of a $\partial X$-segment. 
We also have a {\bf rare} application of the Warner comparison theorem in the proofs of Lemmas 
\ref{lem2.5.new1} and \ref{lem2.5}. 
Notice that Lemmas \ref{lem2.2} and \ref{lem2.6} are indispensable for us to prove 
Lemma on thin open triangles.\par
In Section \ref{sec:LTOT}, we prove Lemma on thin open triangles, 
using Lemmas \ref{lem2.2} and \ref{lem2.6}.\par 
In Section \ref{sec:Rem}, we see, without curvature assumption, 
that the opposite side of any open triangle to $\partial X$ does not meet $\partial X$, 
if $\partial X$ is convex.\par
In Section \ref{sec:AC}, we establish the Alexandrov convexity (Lemma \ref{lem4.4}). 
In the proof of Lemma \ref{lem4.4}, 
we may understand that it is a very important property that 
the opposite side of an open triangle to the boundary in a model surface is 
unique (i.e., we can not prove the equation (\ref{lem4.4-1}) in the proof of Lemma \ref{lem4.4} 
without this property). 
In order to prove Lemma \ref{lem4.4}, we have to treat a non-differentiable Lipschitz function. 
It follows from Dini's theorem (\cite{D}) that, for any Lipschitz function $f$ on $[a, b]$, 
$f$ is differentiable almost everywhere, and 
\[
\int_{a}^{b} f'(t) dt = f(b) -f(a)
\]
holds. Note that the Cantor-Lebesgue function $g$ on $[0, 1]$ is differentiable almost everywhere, 
but 
\[
0 = \int_{0}^{1} g'(t) dt < g(1) -g(0) = 1.
\]
Cohn\,-Vossen applied in \cite{CV1} and \cite{CV2} 
the properties above to global differential geometry.\par
In Section \ref{sec:TCTproof}, we prove our main theorem, 
the Toponogov comparison theorem for open triangles, 
by using new techniques established in \cite[Section 4]{KT2} and 
Lemma \ref{lem4.4}. 

\section{The focal cut locus of $\partial X$}\label{sec:FCL}

Our purpose of this section is to investigate the relationship 
between minimal geodesic segments in a complete connected 
Riemannian manifold with smooth boundary and the focal cut locus 
of the boundary (see Lemma \ref{lem2.2}). 
It will be clarified, by using Lemma \ref{lem2.2}, 
in Section \ref{sec:LTOT} that the cut locus of the manifold 
is not an obstruction at all when we draw a corresponding open triangle 
in a model surface for each open triangle in the manifold. 

\bigskip

Throughout this section, let $(X, \partial X)$ denote a complete connected 
Riemannian $n$-manifold $X$ with smooth boundary $\partial X$.

\bigskip

First, we will recall the definitions of 
$\partial X$-Jacobi fields, focal loci of $\partial X$, and cut loci of $\partial X$, 
which are used throughout this article.

\begin{definition}{\bf ($\partial X$-Jacobi field)} 
Let $\mu :[0, \infty) \rightarrow X$ be a unit speed geodesic emanating perpendicularly from $\partial X$. 
A Jacobi field $J_{\partial X}$ along $\mu$ is called a {\em $\partial X$-Jacobi field}, 
if $J_{\partial X}$ satisfies 
\[
J_{\partial X} (0) \in T_{\mu(0)} \partial X, \quad 
J'_{\partial X} (0) + A_{\mu'(0)} (J_{\partial X} (0)) \in (T_{\mu(0)} \partial X)^{\bot}.
\]
Here $J'$ denotes the covariant derivative of $J$ along $\mu$, 
and $A_{\mu'(0)}$ denotes the shape operator of $\partial X$. 
\end{definition}

\begin{definition}{\bf (Focal locus of $\partial X$)} A point $\mu(t_{0})$, $t_{0} \not= 0$, 
is called a {\em focal point of $\partial X$} along a unit speed geodesic $\mu :[0, \infty) \rightarrow X$ 
emanating perpendicularly from $\partial X$, 
if there exists a non-zero $\partial X$-Jacobi field $J_{\partial X}$ along $\mu$ such that 
$J_{\partial X} (t_{0}) = 0$. 
The {\em focal locus $\Focal (\partial X)$ of $\partial X$} is the union of the focal points of 
$\partial X$ along all of the unit speed geodesics emanating perpendicularly 
from $\partial X$.
\end{definition}

\begin{definition}{\bf (Cut locus of $\partial X$)}
Let $\mu :[0, l_{0}] \rightarrow X$ be a $\partial X$-segment. 
The end point $\mu(l_{0})$ of $\mu ([0, l_{0}])$ 
is called a {\em cut point of} $\partial X$ along $\mu$, 
if any extended geodesic $\bar{\mu} :[0,l_{1}] \lra X$ of $\mu$, $l_{1} > l_{0}$, 
is not a $\partial X$-segment anymore. 
The {\em cut locus $\Cut (\partial X)$ of $\partial X$} is 
the union of the cut points of $\partial X$ along all 
of the $\partial X$-segments.
\end{definition}

Set 
\[
\FC(\partial X) := \Focal (\partial X) \cap \Cut (\partial X).
\]
We then call $\FC(\partial X)$ the {\em focal cut locus of} $\partial X$.

\bigskip
 
From the similar argument in \cite{IT1}, we have the following lemma.

\begin{lemma}{\rm (see \cite[Lemma 2]{IT1})}\label{lem2.1}
The Hausdorff dimension of $\FC(\partial X)$ is at most $n - 2$. 
In particular, $\cH_{n -1} (\FC(\partial X)) = 0$. 
Here $\cH_{n - 1}$ denotes the $(n - 1)$-dimensional Hausdorff measure. 
\end{lemma}

An open neighborhood $U(q)$ of $q \in X$ is called a 
{\em normal convex neighborhood of} $q$, if, for any points $q_{1}, q_{2} \in U(q)$, 
there exists a unique minimal geodesic segment $\sigma$ joining $q_{1}$ to $q_{2}$ 
such that the segment $\sigma$ is contained in $U(q)$. 
Then, the following lemma follows from Lemma \ref{lem2.1}. 

\begin{lemma}\label{lem2.2}
Assume that 
\[
p \not\in \Focal (\partial X), \quad \ q \not\in \Cut (p), \quad and \quad 
\gamma ([0, d(p, q)]) \cap \partial X = \emptyset,
\]
where $\gamma$ denotes the minimal geodesic segment joining $p$ to $q$. 
Then, for each $v \in \Sph^{n-1}_{q} := \{ v \in T_{q} X \ ; \ \| v \| = 1 \}$, 
there exists a sequence 
\[
\left\{
\gamma_{i} : [0,l_{i}] \rightarrow X
\right\}_{i \in \N}
\]
of minimal geodesic segments $\gamma_{i}$ emanating from $p = \gamma_{i}(0)$ 
convergent to $\gamma$ such that 
\[
\gamma_{i}([0,l_{i}]) \cap \FC(\partial X) = \emptyset 
\]
and 
\[
\lim_{i \to \infty} \frac{1}{\| \exp_{q}^{-1}(\gamma_{i}(l_{i})) \|} \exp_{q}^{-1}(\gamma_{i}(l_{i})) 
= v.
\]
Here $\exp_{q}^{-1}$ denotes the local inverse of the $\exp_{q}$ on 
a normal convex neighborhood $U(q)$ of $q$ disjoint from $\partial X$.
\end{lemma}

\begin{proof}
Let $\left\{ q_{j} \right\}_{j \in \N}$ denote a sequence of points $q_{j} \in U(q)$ convergent to $q$ 
such that 
\[
q_{j} \not\in \Cut (p), \quad \alpha_{j} ([0, d(p, q_{j})]) \cap \partial X = \emptyset, 
\]
and 
\[
\lim_{j \to \infty} \frac{1}{\| \exp_{q}^{-1}(q_{j}) \|} \exp_{q}^{-1}(q_{j}) 
= v.
\]
Here $\alpha_{j} : [0, d(p, q_{j})] \rightarrow X$ denotes the minimal geodesic segment emanating from 
$p = \alpha_{j} (0)$ to $q_{j}$. 
We will prove that, for each $q_{j}$, there exists a sequence 
\[
\{
\gamma_{i}^{(j)} : [0, l_{i}^{(j)}] \rightarrow X
\}_{i \in \N}
\]
of minimal geodesic segments $\gamma_{i}^{(j)}$ emanating from $p = \gamma_{i}^{(j)}(0)$ 
convergent to $\alpha_{j}$ such that 
\begin{equation}\label{lem2.2-1}
\gamma_{i}^{(j)}([0, l_{i}^{(j)}]) \cap \FC(\partial X) = \emptyset. 
\end{equation}
It is sufficient to prove the existence of the sequence $\gamma_{i}^{(j)}$ for each $j \in \N$, 
because it is easy to prove the existence of the sequence 
$\{ \gamma_{i} : [0, l_{i}] \rightarrow X \}_{i \in \N}$ in our lemma by taking a subsequence of 
$\{ \gamma_{i}^{(j)} : [0, l_{i}^{(j)}] \rightarrow X \}_{i,\,j \in \N}$.\par
Choose any $q_{j}$ and fix it. 
Since $p$ is not a focal point of $\partial X$, 
there exists a normal convex neighborhood $B_{2\ve}(p)$ of $p$ with radius $2 \ve$ 
such that 
\begin{equation}\label{2009-01-08-1}
\ol{B_{2\ve}(p)} \cap \Focal (\partial X) = \emptyset.
\end{equation}
Since $q_{j}$ is not a cut point of $p$, 
there exist two numbers $l_{j} > d(p, q_{j})$, $\theta_{j} > 0$, 
and a neighborhood $U_{j}$ around $q_{j}$ such that 
$U_{j}$ is diffeomorphic to $V_{\alpha'_{j}(0)}(\theta_{j}) \times (\ve, l_{j})$. 
Here we set 
\[
V_{\alpha'_{j}(0)}(\theta_{j})
:= 
\left\{
w_{j} \in T_{p} X \ ; \ \| w_{j} \| = 1, \ \angle(w_{j}, \alpha'_{j}(0)) < \theta_{j}
\right\}.
\]
Here, the diffeomorphism $\Phi_{j}$ from $V_{\alpha'_{j}(0)}(\theta_{j}) \times (\ve, l_{j})$ onto 
$U_{j}$ is given by 
\[
\Phi_{j} (w_{j}, s) := \exp_{p}(s\,w_{j}).
\] 
Since $\Phi^{-1}_{j}$ is Lipschitz, 
the map $\Pi_{j} := \cP_{j} \circ \Phi^{-1}_{j} : U_{j} \rightarrow V_{\alpha'_{j}(0)}(\theta_{j})$ 
is also Lipschitz, where 
$\cP_{j} : V_{\alpha'_{j}(0)}(\theta_{j}) \times (\ve, l_{j}) \rightarrow V_{\alpha'_{j}(0)}(\theta_{j})$ 
denotes the projection to the first factor. 
Therefore, it follows from Lemma \ref{lem2.1} that 
\[
\cH_{n - 1}(\Pi_{j}(U_{j} \cap \FC(\partial X))) = 0.
\] 
This implies that there exists a sequence $\{ w^{(j)}_{i} \}_{i \in \N}$ 
of elements $w^{(j)}_{i} \in V_{\alpha'_{j}(0)}(\theta_{j})$ convergent 
to $\alpha'_{j}(0)$ such that 
\begin{equation}\label{2009-01-08-2}
w^{(j)}_{i} \not\in \Pi_{j}(U_{j} \cap \FC(\partial X))
\end{equation}
for each $i \in \N$. 
Let $\{l^{(j)}_{i}\}_{i \in \N}$ be a sequence of numbers 
$l^{(j)}_{i} \in (0, l_{j})$ convergent to $d(p, q_{j})$. 
By setting 
\[
\gamma^{(j)}_{i}(s) := \exp_{p}(s\,w^{(j)}_{i}), \quad s \in [0, l^{(j)}_{i}],
\]
for each $i \in \N$, 
it follows from (\ref{2009-01-08-1}) and (\ref{2009-01-08-2}) that 
we get a sequence of minimal geodesic segments $\gamma^{(j)}_{i}$ 
emanating from $p = \gamma^{(j)}_{i}(0)$ 
convergent to $\alpha_{j}$ satisfying (\ref{lem2.2-1}).
$\qedd$
\end{proof}

\section{Length of $\partial X$-segments in variations}\label{sec:LGV}

Our purpose of this section is to prove a comparison theorem (Lemma \ref{lem2.6}) 
of the Rauch type on length of $\partial X$-segments in variations of a $\partial X$-segment, 
by using the second variation formula and the Warner comparison theorem. 
As a result, 
readers might be surprised by, and would realize, as Gromoll once suggested, 
that we may still understand a global matter on a Riemannian manifold by the second variation, 
because Lemma on thin open triangles (Lemma \ref{lem3.8}), 
proved by Lemmas \ref{lem2.2} and \ref{lem2.6}, plays an important role 
in the proof of the Toponogov comparison theorem for open triangles (see Section \ref{sec:TCTproof}).  

\bigskip

Throughout this section, let $(X, \partial X)$ denote a complete connected Riemannian 
$n$-manifold $X$ with smooth {\bf convex} boundary $\partial X$ 
whose radial curvature is bounded from below by the radial curvature function $G$ 
of a model surface $(\wt{X}, \partial \wt{X})$ with its metric (\ref{model-metric}).

\bigskip

Take any point $r \in X \setminus \left( \partial X \cup \Focal (\partial X) \right)$, and fix it. 
Then there exists a positive number 
$\ve_{0} : = \ve_{0}(r)$ such that 
\begin{equation}\label{2.2}
B_{2\ve_{0}}(r) \cap \left( \Focal (\partial X) \cup \partial X \right)= \emptyset,
\end{equation}
where $B_{2\ve_{0}}(r)$ denotes the normal convex neighborhood of $r$ with radius $2\ve_{0}$. 
Take any point $p \in B_{\ve_{0}}(r)$, and fix it. 
Let $\mu : [0, l] \rightarrow X$ denote a $\partial X$-segment to $p = \mu (l)$. 
By (\ref{2.2}), we may find a number $\ve_{1} \in (0, \ve_{0}]$ independent of the choice of $p$ 
and an open neighborhood $\cU$ around $l \mu'(0)$ such that 
\[
\exp^{\bot} : \cU \rightarrow B_{\ve_{1}}(p)
\] 
is a diffeomorphism. 
Here $\exp^{\bot}$ denotes the normal exponential map 
on the normal bundle of $\partial X$. 
Let $\xi : \R \rightarrow \Sph^{n-1}_{p}$ be a unit speed geodesic on 
$\Sph^{n-1}_{p}$ emanating from $\mu'(l) = \xi(0)$, 
where $\Sph^{n-1}_{p} := \{ v \in T_{p} X \ ; \ \| v \| = 1 \}$. 
Notice that $\angle (\mu'(l), \xi (\theta)) = | \theta |$ for all $\theta \in [-\pi, \pi]$. 
From now on, we assume that the curve $\xi$ and its parameter value $\theta \in [-\pi, \pi]$ 
are also fixed. 
Then, we get a minimal geodesic segment $c$ emanating from $p = c(0)$ defined by 
\[
c(s) := \exp_{p} (s\,\xi (\theta))
\]
for all $s \in (-\ve_{1}, \ve_{1})$. 
Thus, we get a geodesic variation 
$\varphi : [0,l] \times (- \ve_{1}, \ve_{1}) \rightarrow X$ of $\mu$ defined by 
\[
\varphi (t, s) := \exp^{\bot} \left( \frac{t}{l}\,v(s) \right),
\]
where we set 
$v(s) := \left( \exp^{\bot} |_{\cU} \right)^{-1}(c(s))$. 
For each $s \in (-\ve_{1}, \ve_{1})$, 
$c(s)$ is joined by a geodesic segment $\varphi_{s}(\,\cdot\,) := \varphi (\,\cdot\,,s)$ emanating perpendicularly from $\partial X$. 
By setting 
\[
J_{\partial X} (t) := \frac{\partial \varphi}{\partial s} (t, 0),
\]
we get a $\partial X$-Jacobi field  $J_{\partial X}$ along $\mu$. 
It is clear that 
\begin{equation}\label{2proof-1}
J_{\partial X} (l) = c'(0).
\end{equation}

Then, we first get the following lemma.

\begin{lemma}\label{lem2.3.new1}
For each $ t \in [0, l]$, an orthogonal component $Y_{\partial X}(t)$ of  $J_{\partial X}(t)$ 
with respect to $\mu'(t)$ is given by 
\[
Y_{\partial X} (t) := J_{\partial X} (t) - \frac{\cos \theta}{l}\,t\,\mu'(t).
\]
\end{lemma}

\begin{proof}
Since $J_{\partial X}$ is a Jacobi field along $\mu$, 
there exist constant numbers $a$ and $b$ 
satisfying 
\[
\big\langle 
J_{\partial X} (t), \mu' (t)
\big\rangle 
= a t + b
\]
for all $t \in [0, l]$. 
Since $J_{\partial X} (0)$ is orthogonal to $\mu'(0)$, 
we see $b = 0$. 
Furthermore, by (\ref{2proof-1}), we see 
\[
a = \frac{\cos \theta}{l}.
\]
Thus, we get 
\[
\big\langle 
J_{\partial X} (t), \mu' (t)
\big\rangle 
= \frac{\cos \theta}{l}\,t
\]
for all $t \in [0, l]$. 
Hence, the Jacobi field $Y_{\partial X}$ 
along $\mu$ defined by 
\[
Y_{\partial X} (t) := J_{\partial X} (t) - \frac{\cos \theta}{l}\,t\,\mu'(t)
\]
is orthogonal to $\mu'(t)$ on $[0, l]$. 
$\qedd$
\end{proof}

\bigskip

In this article, we denote by 
\[
\cI_{\partial X}^{l} (V, W) := I_{l} (V, W) - 
\big\langle 
A_{\mu'(0)} (V (0)), W (0)
\big\rangle 
\]
the index form with respect to $\mu|_{[0, \,l]}$ for piecewise $C^{\infty}$ 
vector fields $V, W$ along $\mu|_{[0, \,l]}$, where we set
\[
I_{l} (V, W) 
:= 
\int^{l}_{0} 
\left\{ 
\big\langle 
V', W'
\big\rangle
-
\big\langle 
R(\mu', V) \mu', W
\big\rangle
\right\}
dt,
\]
which is a symmetric bilinear form. 
The following lemma is clear from the first and second variation formulas 
and Lemma \ref{lem2.3.new1}.

\begin{lemma}\label{lem2.3}The equalities $L'(0) = \cos \theta$ and 
$L''(0) = \cI_{\partial X}^{l} (Y_{\partial X}, Y_{\partial X})$ hold. 
Here $L(s)$ denotes the length of the geodesic segment $\varphi_{s}(\,\cdot\,)$ 
emanating perpendicularly from $\partial X$. 
\end{lemma}

Now, choose any sufficiently small number $\lambda > 0$ and fix it. 
Let  $(\wt{X}_{\lambda}, \partial \wt{X}_{\lambda})$ denote a model surface 
with its metric 
\[
\tilde{g}_{\lambda} = d\tilde{x}^2 + m_{\lambda}(\tilde{x})^{2} d\tilde{y}^2
\]
on $[0, \infty) \times \R$. 
Here the positive smooth function $m_{\lambda}$ satisfies the differential equation 
\[
m_{\lambda}'' + (G - \lambda) m_{\lambda} = 0, \quad m_{\lambda}(0) = 1, \ m_{\lambda}'(0) = 0,
\]
where 
$G$ denotes the radial curvature function of $(\wt{X}, \partial \wt{X})$. 
Thus, the radial curvature of $(X, \partial X)$ is greater than $G_{\lambda} := G - \lambda$. 
Take any point $\tilde{p}$ in $\wt{X}_{\lambda} \setminus \partial \wt{X}_{\lambda}$ 
satisfying 
\[
d(\partial \wt{X}_{\lambda}, \tilde{p}) = d(\partial X, p) = d(\partial X, \mu (l)) = l.
\] 
Throughout this section, we fix $\tilde{p}$.\par 
Let $\tilde{\mu}_{\lambda} : [0,l] \rightarrow \wt{X}_{\lambda}$ denote 
a $\partial \wt{X}_{\lambda}$-segment to $\tilde{p}$, 
and let $\wt{E}_{\lambda}$ denote a unit parallel vector field along $\tilde{\mu}_{\lambda}$ 
orthogonal to $\tilde{\mu}_{\lambda}$. 
Then, we define a $\partial \wt{X}_{\lambda}$-Jacobi field $\wt{Z}_{\lambda}$ 
along $\tilde{\mu}_{\lambda}$ by 
\[
\wt{Z}_{\lambda}(t) := \frac{1}{m_{\lambda}(l)} m_{\lambda} (t) \wt{E}_{\lambda}(t).
\]
Furthermore, by the same definition, we denote also by 
$I_{l} (\,\cdot\,, \,\cdot\,)$ 
the symmetric bilinear form for piecewise $C^{\infty}$ vector fields along 
$\tilde{\mu}_{\lambda}|_{[0, \,l]}$. 
Then, we have the following lemma. 

\begin{lemma}\label{lem2.5.new1}
\[
I_{l} (\wt{Z}_{\lambda}, \wt{Z}_{\lambda}) 
\ge \cI_{\partial X}^{l} (Z_{\partial X}, Z_{\partial X}) 
+ \frac{\lambda}{m_{\lambda}(l)^{2}} \int_{0}^{l} m_{\lambda}(t)^{2}\,dt
\]
holds for all $\partial X$-Jacobi fields $Z_{\partial X}$ along $\mu$ orthogonal to $\mu$ with 
$\|Z_{\partial X} (l)\| = 1$.
\end{lemma}

\begin{proof}
We can prove this lemma by the argument in the proof of the Warner comparison theorem \cite{W}. 
For completeness, we will give a proof here. 
Let $E$ be a unit parallel vector field along $\mu$ orthogonal to $\mu$ such that 
\[
E(l) = Z_{\partial X}(l),
\]
where $Z_{\partial X}$ denotes a $\partial X$-Jacobi field along $\mu$ orthogonal to $\mu$. 
Set
\[
W(t) := \frac{1}{m_{\lambda}(l)} m_{\lambda} (t) E(t).
\]
Since 
$
K_{X}(\sigma_{t}) 
\ge G (\tilde{\mu}(t)) 
> G_{\lambda}(\tilde{\mu}_{\lambda}(t)) 
= G (\tilde{\mu}(t)) - \lambda$, 
we have
\begin{align}\label{lem2.5.new1-1}
I_{l} (\wt{Z}_{\lambda}, \wt{Z}_{\lambda}) 
&= \int^{l}_{0} 
\left\{ 
\left\langle 
\wt{Z}'_{\lambda}, \wt{Z}'_{\lambda}
\right\rangle
-
G_{\lambda}(\tilde{\mu}_{\lambda}(t)) 
\left\|
\wt{Z}_{\lambda}
\right\|^{2}
\right\}
dt\\[2mm]
&= \int^{l}_{0} 
\left\{ 
\left\langle 
W', W'
\right\rangle
-
(G (\tilde{\mu}(t)) - \lambda)
\left\|
W
\right\|^{2}
\right\}
dt\notag\\[2mm]
&\ge \int^{l}_{0} 
\left\{ 
\left\langle 
W', W'
\right\rangle
-
K_{X} (\sigma_{t})
\left\|
W
\right\|^{2}
\right\}
dt
+ 
\lambda 
\int^{l}_{0} 
\left\|
W
\right\|^{2}
dt
\notag\\[2mm]
&= 
I_{l} (W, W) 
+ 
\frac{\lambda}{m_{\lambda}(l)^{2}} \int_{0}^{l}m_{\lambda}(t)^{2}\,dt.\notag
\end{align}
Since $Z_{\partial X}$ is the $\partial X$-Jacobi field with 
$Z_{\partial X}(l) = E(l) = W(l)$, 
it follows from \cite[Lemma 2.10 in Chapter III]{S} that 
\begin{equation}\label{lem2.5.new1-2}
I_{l} (W, W) - \big\langle A_{\mu'(0)} (W(0) ), W(0) \big\rangle
=
\cI_{\partial X}^{l} (W, W)
\ge 
\cI_{\partial X}^{l} (Z_{\partial X}, Z_{\partial X}).
\end{equation}
Since $ \big\langle A_{\mu'(0)} (W(0) ), W(0) \big\rangle \ge 0$, 
we get, by (\ref{lem2.5.new1-1}) and (\ref{lem2.5.new1-2}), 
\begin{align}
I_{l} (\wt{Z}_{\lambda}, \wt{Z}_{\lambda}) 
&\ge
\cI_{\partial X}^{l} (Z_{\partial X}, Z_{\partial X})
+\big\langle A_{\mu'(0)} (W(0) ), W(0) \big\rangle
+ 
\frac{\lambda}{m_{\lambda}(l)^{2}} \int_{0}^{l}m_{\lambda}(t)^{2}\,dt\notag\\[2mm]
&\ge
\cI_{\partial X}^{l} (Z_{\partial X}, Z_{\partial X})
+ 
\frac{\lambda}{m_{\lambda}(l)^{2}} \int_{0}^{l}m_{\lambda}(t)^{2}\,dt\notag.
\end{align}
$\qedd$
\end{proof}

\bigskip

Let $\tilde{c}_{\lambda} : (-\ve_{1}, \ve_{1}) \rightarrow \wt{X}_{\lambda}$ denote 
the minimal geodesic segment emanating from $\tilde{p} = \tilde{c}_{\lambda}(0)$ corresponding to 
the minimal geodesic segment $c(s) = \exp_{p} (s \, \xi (\theta)), s \in (-\ve_{1}, \ve_{1})$ 
in $B_{\ve_{1}}(p) \subset X$. 
Without loss of generality, 
we may assume that $B_{\ve_{1}} (\tilde{p}) \cap \partial \wt{X}_{\lambda} = \emptyset$. 
We consider a geodesic variation 
$\tilde{\varphi}^{(\lambda)} : [0,l] \times (- \ve_{1}, \ve_{1}) \rightarrow \wt{X}_{\lambda}$ 
of $\tilde{\mu}_{\lambda}$ defined by 
\[
\tilde{\varphi}^{(\lambda)} (t, s) := \exp^{\bot} \left( \frac{t}{l}\,\tilde{v}_{\lambda}(s) \right),
\]
where we set 
$\tilde{v}_{\lambda}(s) := \left( \exp^{\bot} \right)^{-1}(\tilde{c}_{\lambda}(s))$. 
By setting 
\[
\wt{J}_{\lambda} (t) := \frac{\partial \tilde{\varphi}^{(\lambda)}}{\partial s} (t, 0),
\]
we get a $\partial \wt{X}_{\lambda}$-Jacobi field $\wt{J}_{\lambda}$ 
along $\tilde{\mu}_{\lambda}$. 
As well as above, $\wt{J}_{\lambda} (l) = \tilde{c}_{\lambda}'(0)$ holds, 
and  the Jacobi field $\wt{Y}_{\lambda}$ 
along along $\tilde{\mu}_{\lambda}$ defined by 
\[
\wt{Y}_{\lambda} (t) 
:= 
\wt{J}_{\lambda} (t) 
- 
\frac{\cos \theta}{l}\,t\,\tilde{\mu}_{\lambda}'(t)
\]
is orthogonal to $\tilde{\mu}_{\lambda}'(t)$ on $[0, l]$. 

\begin{lemma}\label{lem2.5}
There exists a number $\lambda_{0} := \lambda_{0}(l_{0}, \ve_{0})> 0$ depending 
on $l_{0}$ and $\ve_{0}$ such that, for any $\lambda \in (0, \lambda_{0})$, 
any unit speed geodesic $\xi$ on $\Sph^{n-1}_{p}$ emanating from $\mu'(l)$, 
and any $\theta \in (0, \pi)$, the inequality 
\[
I_{l} (\wt{Y}_{\lambda}, \wt{Y}_{\lambda}) 
- \cI_{\partial X}^{l} (Y_{\partial X}, Y_{\partial X}) 
\ge 
\lambda \,C_{1} \sin^{2} \theta
\]
holds. 
Here $C_{1}$ is a constant number given by 
\[
C_{1} := \frac{1}{2 m(l_{0})^{2}} \int_{0}^{l_{0}} m(t)^{2} \,dt,
\]
where $l_{0} := d(\partial X, r)$.
\end{lemma}

\begin{proof}
Since 
\[
\wt{Y}_{\lambda} (l) = \tilde{c}_{\lambda}'(0) 
- \cos \theta \, \tilde{\mu}_{\lambda}'(l) = \pm \sin \theta \cdot \wt{E}_{\lambda} (l) 
= \pm \sin \theta \cdot \wt{Z}_{\lambda}(l),
\]
and since $\wt{Y}_{\lambda}$ is a $\partial \wt{X}_{\lambda}$-Jacobi field orthogonal to 
$\tilde{\mu}_{\lambda}$, we see 
\[
\wt{Y}_{\lambda}(t) 
= \pm \sin \theta \cdot\wt{Z}_{\lambda}(t)
\]
on $[0, l]$. Notice that any $\partial \wt{X}_{\lambda}$-Jacobi field orthogonal to 
$\tilde{\mu}_{\lambda}$ is equal to $a \wt{Z}_{\lambda} (t)$, $a \in \R$. 
Hence, we have 
\begin{equation}\label{lem2.5-2}
I_{l} (\wt{Y}_{\lambda}, \wt{Y}_{\lambda}) 
= \sin^{2}\theta  \cdot I_{l} (\wt{Z}_{\lambda}, \wt{Z}_{\lambda}).
\end{equation}
Similarly, we have, by Lemma \ref{lem2.3.new1},  
\[
Y_{\partial X}(t) 
= \sin \theta \cdot Z_{\partial X}(t) 
\]
for some $\partial X$-Jacobi field $Z_{\partial X}$ along $\mu$ 
orthogonal to $\mu$ with $\|Z_{\partial X} (l)\| = 1$.
Hence, we have 
\begin{equation}\label{lem2.5-3}
\cI_{\partial X}^{l} (Y_{\partial X}, Y_{\partial X}) 
= \sin^{2} \theta \cdot \cI_{\partial X}^{l} (Z_{\partial X}, Z_{\partial X}).
\end{equation}
By combining (\ref{lem2.5-2}) and (\ref{lem2.5-3}), 
we get, by Lemma \ref{lem2.5.new1}, 
\begin{align}\label{lem2.5-4}
I_{l} (\wt{Y}_{\lambda}, \wt{Y}_{\lambda}) 
- \cI_{\partial X}^{l} (Y_{\partial X}, Y_{\partial X}) 
&=
\sin^{2} \theta \left\{
I_{l} (\wt{Z}_{\lambda}, \wt{Z}_{\lambda}) 
- \cI_{\partial X}^{l} (Z_{\partial X}, Z_{\partial X})
\right\}\\[2mm]
&\ge
\frac{\lambda \sin^{2} \theta}{m_{\lambda}(l)^{2}} \int_{0}^{l}m_{\lambda}(t)^{2}\,dt.\notag
\end{align}
On the other hand, since $\lim_{\lambda \downarrow 0} m_{\lambda} (t) = m(t)$ 
and $|l - l_{0}| < \ve_{0}$, 
we may find a number $\lambda_{0} > 0$ such that
\begin{align}\label{lem2.5-5}
\frac{1}{m_{\lambda}(l)^{2}} \int_{0}^{l}m_{\lambda}(t)^{2}\,dt 
> 
\frac{1}{2m(l_{0})^{2}} \int_{0}^{l_{0}}m(t)^{2}\,dt
\end{align}
for all $\lambda \in (0, \lambda_{0})$. 
From (\ref{lem2.5-4}) and (\ref{lem2.5-5}), we have proved this lemma.
$\qedd$
\end{proof}

\begin{lemma}{\bf (Key lemma)}\label{lem2.6} 
For each $\lambda \in (0, \lambda_{0})$, there exists a number 
$\delta_{1} := \delta_{1}(\lambda) \in (0, \ve_{0})$ such that, for any $p \in B_{\ve_{0}}(r)$, 
any unit speed geodesic $\xi$ on $\Sph_{p}^{n -1}$ emanating from $\mu'(l)$, 
any $\theta \in [0, \pi]$, and any $\lambda \in (0, \lambda_{0})$, the inequality 
\[
L(s) \le \wt{L}_{\lambda} (s)
\]
holds for all $s \in [0, \delta_{1}]$, and the equality occurs if and only if $s = 0$, $\theta = 0$, 
or $\theta = \pi$. Here $\wt{L}_{\lambda} (s)$ denotes the length of the geodesic segment 
$\tilde{\varphi}^{(\lambda)}_{s}(\,\cdot\,) = \tilde{\varphi}^{(\lambda)} (\,\cdot\,,s)$ 
emanating perpendicularly from $\partial \wt{X}_{\lambda}$ to $\tilde{c}_{\lambda}(s)$.  
\end{lemma}

\begin{proof} 
Although the angle $\theta$ has been fixed in the arguments 
above of this section, 
we consider here that $\theta$ is a variable. 
Hence, we denote $L(s)$ by $L(s, \theta)$, 
which is a smooth function of two variables $s$ and $\theta$ 
and depends smoothly on $p$, $\xi(0)$, and $\xi'(0)$. 
Furthermore, we define the reminder term $\cR (s, \theta)$ of the Taylor expansion of 
$L(s, \theta)$ about $s = 0$ by 
\begin{equation}\label{lem2.6-1} 
\cR (s, \theta) := L(s, \theta) - \left\{ L(0, \theta) + L'(0, \theta)s + \frac{1}{2!} L''(0, \theta)s^{2} \right\},
\end{equation}
where we set 
\[
L'(0, \theta) := \frac{\partial L}{\partial s} (0, \theta)
\quad {\rm and} \quad 
L''(0, \theta) := \frac{\partial^{2} L}{\partial s^{2}} (0, \theta).
\]
From (\ref{lem2.6-1}), Lemma \ref{lem2.3}, and the equation (\ref{lem2.5-3}) in the proof of 
Lemma \ref{lem2.5}, we have 
\begin{align}
L(s, \theta) 
&= l + s \cos \theta 
+ \frac{s^{2}}{2} \, \cI_{\partial X}^{l} (Y_{\partial X}, Y_{\partial X})
+ \cR (s, \theta)\label{lem2.6-1.5new}\\[2mm]
&
= l + s \cos \theta 
+ \frac{s^{2} \sin^{2} \theta}{2} \, \cI_{\partial X}^{l} (Z_{\partial X}, Z_{\partial X})
+ \cR (s, \theta).\label{lem2.6-1.6new} 
\end{align}
It is clear that 
\[
\cR (0, \theta) 
= \frac{\partial \cR}{\partial s} (0, \theta) 
= \frac{\partial^{2} \cR}{\partial s^{2}} (0, \theta)
= 0.
\]
Hence, there exists a smooth function $\cR_{1} (s, \theta)$ depending smoothly 
on $p$, $\xi(0)$, and $\xi'(0)$ such that 
\begin{equation}\label{lem2.6-2} 
\cR (s, \theta) = \cR_{1} (s, \theta) s^{3}.
\end{equation}
Since 
$B_{2\ve_{0}} (r) \cap \Focal (\partial X) = \emptyset$, 
the geodesic $\varphi_{s} (\,\cdot\,)$ is locally minimal for each $s \in (-\ve_{1}, \ve_{1})$. 
Hence, we may assume that the triangle inequalities 
\begin{equation}\label{lem2.6-3} 
L(s, \theta) \le l + s = L(s, 0)
\end{equation}
and 
\begin{equation}\label{lem2.6-4} 
L(s, \pi - \theta) \ge l - s = L(s, \pi)
\end{equation}
hold for all sufficiently small $|\theta|$ and all $s \in [0, \ve_{1})$. 
The equations (\ref{lem2.6-3}) and (\ref{lem2.6-4}) mean that, for each $s \in (0, \ve_{1})$, 
the function $L(s, \,\cdot\,)$ attains a local maximum (resp. minimum) at $\theta = 0$ 
(resp. $\theta = \pi$). 
Hence, by (\ref{lem2.6-1.6new}) and (\ref{lem2.6-2}),
\begin{equation}\label{lem2.6-5} 
\frac{\partial \cR_{1}}{\partial \theta} (s, 0) = \frac{\partial \cR_{1}}{\partial \theta} (s, \pi) = 0 
\end{equation}
for each $s \in [0, \ve_{1})$. 
Since $\cR_{1}(s, 0) = \cR_{1}(s, \pi) = 0$ holds on $[0, \ve_{1})$, we see, by (\ref{lem2.6-5}), 
that there exists a smooth function $\cR_{2}(s,\theta)$ such that  
\begin{equation}\label{lem2.6-6} 
\cR_{1}(s, \theta) = \cR_{2}(s, \theta) \theta^{2} (\pi - \theta)^{2}.
\end{equation}
By (\ref{lem2.6-2}) and  (\ref{lem2.6-6}), we have 
\begin{equation}\label{lem2.6-7} 
\cR(s, \theta) = \cR_{2}(s, \theta) \theta^{2} (\pi - \theta)^{2} s^{3}
\end{equation}
for all $\theta \in [0, \pi]$ and all $s \in [0, \ve_{1})$. 
On the other hand, 
since $\cR_{2}$ depends continuously on $p$, $\xi(0)$, and $\xi'(0)$, 
there exists a constant $C_{2} > 0$ such that  
\begin{equation}\label{lem2.6-8} 
|\cR_{2}(s, \theta)| \le C_{2}
\end{equation}
holds for all $p \in B_{\ve_{0}}(r)$, 
all $\xi$ on $\Sph^{n -1}_{p}$, all $\theta \in [0, \pi]$, and all $s \in [0, \ve_{1}/2]$. 
Thus, by (\ref{lem2.6-7}) and (\ref{lem2.6-8}), we obtain 
\begin{equation}\label{lem2.6-9} 
|\cR(s, \theta)| \le C_{2}\,\theta^{2} (\pi - \theta)^{2} s^{3}
\end{equation}
for all $p \in B_{\ve_{0}}(r)$, all $\xi$ on $\Sph^{n -1}_{p}$, 
all $\theta \in [0, \pi]$ and all $s \in [0, \ve_{1} / 2]$. 
Combining (\ref{lem2.6-1.5new}) and (\ref{lem2.6-9}), 
we get 
\begin{equation}\label{lem2.6-10} 
L(s, \theta) \le l + s \cos \theta + \frac{s^{2}}{2} \cI_{\partial X}^{l} (Y_{\partial X}, Y_{\partial X}) + 
C_{2}\,\theta^{2} (\pi - \theta)^{2} s^{3}.
\end{equation}
By applying the same argument above for $\wt{L}_{\lambda} (s) = \wt{L}_{\lambda} (s, \theta)$, 
there exists a constant $C_{3} > 0$ such that 
\begin{equation}\label{lem2.6-11} 
\wt{L}_{\lambda} (s, \theta) 
\ge 
l + s \cos \theta + \frac{s^{2}}{2} 
I_{l} (\wt{Y}_{\lambda}, \wt{Y}_{\lambda}) 
- C_{3}\,\theta^{2} (\pi - \theta)^{2} s^{3}
\end{equation}
holds for all $\theta \in [0, \pi]$ and all $s \in [0, \ve_{1}/2]$. 
From Lemma \ref{lem2.5}, (\ref{lem2.6-10}), and (\ref{lem2.6-11}), 
it follows that 
\begin{align}\label{lem2.6-12} 
\wt{L}_{\lambda} (s, \theta) - L(s, \theta) 
&\ge
\frac{s^{2}}{2} 
\left\{
I_{l} (\wt{Y}_{\lambda}, \wt{Y}_{\lambda}) 
- \cI_{\partial X}^{l} (Y_{\partial X}, Y_{\partial X}) 
\right\} - (C_{3} + C_{2})\theta^{2} (\pi - \theta)^{2} s^{3}\\[2mm]
&\ge
\frac{\lambda C_{1} \sin^{2} \theta}{2} s^{2}
- (C_{3} + C_{2})\theta^{2} (\pi - \theta)^{2} s^{3}\notag\\[2mm]
&\ge
\frac{\lambda C_{1} \sin^{2} \theta}{2} s^{2}
- 2C_{4}\,\theta^{2} (\pi - \theta)^{2} s^{3}\notag
\end{align}
holds for all $\theta \in [0, \pi]$ and all $s \in [0, \ve_{1}/2]$. 
Here we set $C_{4} := \max\{ C_{2}, C_{3} \}$.\par 
Since 
\[
\frac{x}{\sin x} < \frac{\pi}{2}
\]
for all $x \in (0, \pi/2)$, 
\begin{equation}\label{lem2.6-13} 
\frac{\theta}{\sin \theta} \cdot (\pi - \theta) 
< \frac{\pi}{2} \cdot (\pi - \theta)
< \frac{\pi^{2}}{2}
\end{equation}
holds on $(0, \pi/ 2)$, and 
\begin{equation}\label{lem2.6-14} 
\frac{\pi - \theta}{\sin \theta} \cdot \theta = \frac{\pi - \theta}{\sin (\pi - \theta)} \cdot \theta 
< \frac{\pi}{2} \cdot \theta 
< \frac{\pi^{2}}{2}
\end{equation}
also holds on $(\pi / 2, \pi)$. 
Hence, by (\ref{lem2.6-13}) and (\ref{lem2.6-14}), we see 
\begin{equation}\label{lem2.6-15} 
\frac{\theta (\pi - \theta)}{\sin \theta} < \frac{\pi^{2}}{2}
\end{equation}
on $(0, \pi)$. 
If we define 
\[
\delta_{1} 
:= \min 
\left\{ 
 \frac{\ve_{1}}{2}_, \frac{\lambda C_{1}}{\pi^{4} C_{4}}
\right\} 
\left(
\le \frac{\ve_{1}}{2} < \ve_{0}
\right),
\]
then, by (\ref{lem2.6-15}), 
\begin{align}\label{lem2.6-16} 
\frac{\lambda C_{1} \sin^{2} \theta}{2} s^{2}
- 2C_{4}\,\theta^{2} (\pi - \theta)^{2} s^{3}
&=
\frac{(s \cdot \sin \theta)^{2}}{2}
\left[
\lambda C_{1}
- 
4 C_{4} 
\left\{
\frac{\theta (\pi - \theta)}{\sin \theta}
\right\}^{2}
s
\right]\\[2mm]
&>
\frac{(s \cdot \sin \theta)^{2}}{2}
\left(
\lambda C_{1}
- 
\pi^{4} C_{4} s
\right)
\notag\\[2mm]
&\ge
0\notag
\end{align}
holds for all $s \in [0, \delta_{1}]$ and all $\theta \in (0, \pi)$. 
Therefore, by (\ref{lem2.6-12}) and (\ref{lem2.6-16}), the proof is completed.
$\qedd$
\end{proof}

\section{Thin open triangles}\label{sec:LTOT}
Throughout this section, let $(\wt{X}, \partial \wt{X})$ denote a model surface 
with its metric (\ref{model-metric}). 

\begin{lemma}\label{lem3.1}
Let $\tilde{\mu} : [0, l] \rightarrow \wt{X}$ be a $\partial \wt{X}$-segment. 
Then, for each 
\[
0 < s < \min \{\inj (\tilde{\mu}(l)), l\},
\]
the function $d(\partial \wt{X}, \exp_{\tilde{\mu}(l)}(s \, \tilde{\xi}(\theta))$ 
is strictly increasing on $[0, \pi]$. Here 
\[
\tilde{\xi} : \R \rightarrow \Sph^{1}_{\tilde{\mu}(l)} 
:= \{ \tilde{v} \in T_{\tilde{\mu}(l)} \wt{X} \ ; \ \| \tilde{v} \| = 1 \}
\] 
denotes a unit speed geodesic segment on $\Sph^{1}_{\tilde{\mu}(l
)}$ 
emanating from $- \tilde{\mu}'(l
) = \tilde{\xi}(0)$. 
\end{lemma}

\begin{proof}
This lemma is clear from the first variation formula. 
$\qedd$
\end{proof}

\medskip

The next lemma is a direct consequence of the Clairaut relation (\cite[Theorem 7.1.2]{SST}) 
and the first variational formula\,:

\begin{lemma}\label{lem4.1}
For each constant $c \ge 0$, and each point $\tilde{p} \in \wt{X}$, 
$d(\tilde{p}, \tilde{\tau}_{c}(s))$ is strictly increasing on $[\tilde{y}(\tilde{p}), \infty)$. 
Here $\tilde{\tau}_{c}(s) := (c, s) \in \wt{X}$ denote the arc of $\tilde{x} = c$.
\end{lemma} 

\medskip

By Lemma \ref{lem4.1}, we have

\begin{lemma}\label{lem3.2}
Let 
${\rm OT}(\partial \wt{X}, \tilde{p}_{1}, \tilde{q}_{1})$ and 
${\rm OT}(\partial \wt{X}, \tilde{p}_{2}, \tilde{q}_{2})$ 
be open triangles in $\wt{X}$ such that 
\begin{equation}\label{lem3.2-length1}
d(\partial \wt{X}, \tilde{q}_{1}) = d(\partial \wt{X}, \tilde{p}_{2}),
\end{equation}
and that 
\begin{equation}\label{lem3.2-angle1}
\angle\,\tilde{q}_{1} + \angle\,\tilde{p}_{2} \le \pi.
\end{equation}
If 
\[
d(\tilde{p}_{1}, \tilde{q}_{1}) + d(\tilde{p}_{2}, \tilde{q}_{2}) < \inj (\tilde{p}_{1}),
\]
then there exists an open triangle ${\rm OT}(\partial \wt{X}, \tilde{p}, \tilde{q})$ such that 
\begin{equation}\label{lem3.2-length2}
d(\partial \wt{X}, \tilde{p}) = d(\partial \wt{X}, \tilde{p}_{1}), \quad 
d(\tilde{p}, \tilde{q}) = d(\tilde{p}_{1}, \tilde{q}_{1}) + d(\tilde{p}_{2}, \tilde{q}_{2}), \quad 
d(\partial \wt{X}, \tilde{q}) = d(\partial \wt{X}, \tilde{q}_{2}),
\end{equation}
and that 
\begin{equation}\label{lem3.2-angle2}
\angle\,\tilde{p}_{1} \ge \angle\,\tilde{p}.
\end{equation}
\end{lemma}

\begin{proof}
Let 
\[
{\rm OT}(\partial \wt{X}, \tilde{p}_{1}, \tilde{q}_{1}) 
:=(\partial \wt{X}, \tilde{p}_{1}, \tilde{q}_{1}\, ; \, 
\tilde{\gamma}_{1}, \tilde{\mu}^{(1)}_{1}, \tilde{\mu}^{(1)}_{2})
\]
and 
\[
{\rm OT}(\partial \wt{X}, \tilde{p}_{2}, \tilde{q}_{2}) 
:= 
(\partial \wt{X}, \tilde{p}_{2}, \tilde{q}_{2}\,;\,
\tilde{\gamma}_{2}, \tilde{\mu}^{(2)}_{1}, \tilde{\mu}^{(2)}_{2})
\] 
be open triangles in $\wt{X}$ satisfying (\ref{lem3.2-length1}) and (\ref{lem3.2-angle1}), and 
we fix them. 
By (\ref{lem3.2-length1}), we may assume that 
${\rm OT}(\partial \wt{X}, \tilde{p}_{2}, \tilde{q}_{2})$ is adjacent to 
${\rm OT}(\partial \wt{X}, \tilde{p}_{1}, \tilde{q}_{1})$ as a common side 
$\tilde{\mu}^{(1)}_{2} = \tilde{\mu}^{(2)}_{1}$, and that 
$\tilde{y}(\tilde{p}_{1}) < \tilde{y}(\tilde{q}_{1}) = \tilde{y}(\tilde{p}_{2}) < \tilde{y}(\tilde{q}_{2})$. 
Choose any number 
\[
a \in (d(\tilde{p}_{1}, \tilde{q}_{1}) + d(\tilde{p}_{2}, \tilde{q}_{2}), \inj (\tilde{p}_{1})),
\]
and fix it. We will introduce geodesic polar coordinates $(r, \theta)$ 
around $\tilde{p}_{1}$ on $B_{a}(\tilde{p}_{1})$ such that $\theta = 0$ on 
$\tilde{\mu}^{(1)}_{1} \cap B_{a}(\tilde{p}_{1})$, and that 
$0 < \theta (\tilde{q}_{2}) \le \theta (\tilde{q}_{1}) \le \pi$. 
Notice that 
\[
\tilde{q}_{2} \in B_{a}(\tilde{p}_{1}).
\]
In fact, from the triangle inequality, we have 
\[
d(\tilde{p}_{1}, \tilde{q}_{2}) 
\le d(\tilde{p}_{1}, \tilde{q}_{1}) + d(\tilde{q}_{1}, \tilde{q}_{2}) 
= d(\tilde{p}_{1}, \tilde{q}_{1}) + d(\tilde{p}_{2}, \tilde{q}_{2}) < a.
\]
Since there is nothing to prove if $\angle\,\tilde{q}_{1} + \angle\,\tilde{p}_{2} = \pi$, 
we may assume, by (\ref{lem3.2-angle1}), that 
\begin{equation}\label{lem3.2-1}
\angle\,\tilde{q}_{1} + \angle\,\tilde{p}_{2} < \pi.
\end{equation}
Hence, $\tilde{q}_{2}$ is in $\cA(\tilde{p}_{1})$,
where $\cA(\tilde{p}_{1})$ is a domain defined by 
\[
\cA(\tilde{p}_{1}) := B_{a}(\tilde{p}_{1}) \cap \theta^{-1} (0, \theta (\tilde{q}_{1})).
\]
Let $\tilde{\tau} : [\tilde{y}(\tilde{q}_{2}), \infty) \rightarrow \wt{X}$ be an arc of 
$\tilde{x} = \tilde{x}(\tilde{q}_{2})$ emanating from 
$\tilde{q}_{2} = \tilde{\tau} (\tilde{y}(\tilde{q}_{2})) \in \cA(\tilde{p}_{1})$
given by $\tilde{\tau}(s) := (\tilde{x}(\tilde{q}_{2}), s)$.
By (\ref{lem3.2-1}), we get 
\[
d(\tilde{p}_{1}, \tilde{\tau} (\tilde{y}(\tilde{q}_{2}))) 
< d(\tilde{p}_{1}, \tilde{q}_{1}) + d(\tilde{q}_{1}, \tilde{q}_{2}) 
= d(\tilde{p}_{1}, \tilde{q}_{1}) + d(\tilde{p}_{2}, \tilde{q}_{2}) < a.
\]
Since $\lim_{s \to \infty} d(\tilde{p}_{1}, \tilde{\tau}(s)) = \infty$, 
it follows from the intermediate value theorem that 
there exists a number $s_{0} \in (\tilde{y}(\tilde{q}_{2}), \infty)$ satisfying 
$d(\tilde{p}_{1}, \tilde{\tau} (s_{0})) = a$, 
and furthermore that there exists a number $s_{1} \in (\tilde{y}(\tilde{q}_{2}), s_{0})$ satisfying 
\begin{equation}\label{lem3.2-2}
d(\tilde{p}_{1}, \tilde{\tau} (s_{1})) 
= d(\tilde{p}_{1}, \tilde{q}_{1}) + d(\tilde{q}_{1}, \tilde{q}_{2}) 
= d(\tilde{p}_{1}, \tilde{q}_{1}) + d(\tilde{p}_{2}, \tilde{q}_{2}).
\end{equation}
We will prove that the subarc $\tilde{\tau}|_{[\tilde{y}(\tilde{q}_{2}), \,s_{1}]}$ is 
contained in $\cA(\tilde{p}_{1})$. 
Suppose that there exists a number $s_{2} \in (\tilde{y}(\tilde{q}_{2}), s_{1}]$ such that 
$\tilde{\tau} (s_{2})$ is not in $\cA(\tilde{p}_{1})$.
Since the subarc $\tilde{\tau}|_{[\tilde{y}(\tilde{q}_{2}), \,s_{2}]}$ lies in $B_{a}(\tilde{p}_{1})$, 
there exists $s_{3} \in (\tilde{y}(\tilde{q}_{2}), s_{2}]$ such that 
\begin{equation}\label{lem3.2-3}
\theta (\tilde{\tau} (s_{3})) = \theta (\tilde{q}_{1}).
\end{equation}
Since $\tilde{y}(\tilde{q}_{2}) < s_{3}$, 
we have, by Lemma \ref{lem4.1}, 
\begin{equation}\label{lem3.2-4}
d(\tilde{q}_{1}, \tilde{q}_{2}) < d(\tilde{q}_{1}, \tilde{\tau} (s_{3})).
\end{equation}
By (\ref{lem3.2-3}), we see that 
the geodesic extension $\tilde{\sigma} :[0, d(\tilde{p}_{1}, \tilde{\tau} (s_{3}))] \rightarrow
 \wt{X}$ 
of $\tilde{\gamma}_{1}$ meets $\tilde{\tau}$ at 
$\tilde{\tau} (s_{3}) = \tilde{\sigma} (d(\tilde{p}_{1}, \tilde{\tau} (s_{3})))$. 
Notice that the geodesic segment $\tilde{\sigma}$ is minimal, 
since 
\[
\tilde{\tau} (s_{3}) \in B_{a} (\tilde{p}_{1}) \subset B_{\inj(\tilde{p}_{1})} (\tilde{p}_{1}).
\]
Thus, by (\ref{lem3.2-2}) and (\ref{lem3.2-4}), we have 
\begin{equation}\label{lem3.2-5}
d(\tilde{p}_{1}, \tilde{\tau} (s_{3})) 
= 
d(\tilde{p}_{1}, \tilde{q}_{1}) + d(\tilde{q}_{1}, \tilde{\tau} (s_{3})) 
> 
d(\tilde{p}_{1}, \tilde{q}_{1}) + d(\tilde{q}_{1}, \tilde{q}_{2}) 
= d(\tilde{p}_{1}, \tilde{\tau} (s_{1})). 
\end{equation}
On the other hand, since $s_{3} < s_{1}$, it follows from Lemma \ref{lem4.1} that 
\begin{equation}\label{lem3.2-6}
d(\tilde{p}_{1}, \tilde{\tau} (s_{3})) < d(\tilde{p}_{1}, \tilde{\tau} (s_{1})).
\end{equation}
The equation (\ref{lem3.2-6}) contradicts the equation (\ref{lem3.2-5}). 
Therefore, we have proved that 
the subarc $\tilde{\tau}|_{[\tilde{y}(\tilde{q}_{2}), \,s_{1}]}$ is contained in $\cA(\tilde{p}_{1})$.\par
Since the minimal geodesic segment 
$\tilde{\gamma} : [0, d(\tilde{p}_{1}, \tilde{\tau} (s_{1}))] \rightarrow
 \wt{X}$ joining $\tilde{p}_{1}$ to 
$\tilde{\tau} (s_{1})$ lies in the closure of $\cA(\tilde{p}_{1})$, the inequality 
\[
\angle\,\tilde{p}_{1} 
\ge 
\angle 
\left( 
\tilde{\gamma}'(0), - \frac{d\tilde{\mu}_{1}^{(1)}}{dt}(d(\partial \wt{X}, \tilde{p}_{1}))
\right)
\]
holds. 
Hence, it is clear that the open triangle 
${\rm OT}(\partial \wt{X}, \tilde{p}, \tilde{q}) := (\partial \wt{X}, \tilde{p}, \tilde{q}\,;\,\partial \wt{X}, \tilde{p}_{1}, \tilde{\tau}(s_{1}))$ 
satisfies (\ref{lem3.2-length2}) and (\ref{lem3.2-angle2}) in our lemma.
$\qedd$
\end{proof}

\bigskip

Hereafter, let $(X, \partial X)$ be a complete connected Riemannian 
$n$-manifold $X$ with smooth {\bf convex} boundary $\partial X$ 
whose radial curvature is bounded from below by that of $(\wt{X}, \partial \wt{X})$.\par 
Let $\lambda_{0}$ denote the positive number guaranteed in Lemma \ref{lem2.5}. 
Choose any number $\lambda \in (0, \lambda_{0})$ and fix it. 
In the following, for the $\lambda$, 
we also denote by $(\wt{X}_{\lambda}, \partial \wt{X}_{\lambda})$ 
a model surface with its metric 
\[
\tilde{g}_{\lambda} = d\tilde{x}^2 + m_{\lambda}(\tilde{x})^{2} d\tilde{y}^2
\] 
on $[0, \infty) \times \R$. 
Here the positive smooth function $m_{\lambda}$ satisfies the differential equation 
\[
m_{\lambda}'' + (G - \lambda) m_{\lambda} = 0
\] 
with initial conditions $m_{\lambda}(0) = 1$ and $m_{\lambda}'(0) = 0$, 
where $G$ denotes the radial curvature function of $(\wt{X}, \partial \wt{X})$. 
Then, the next lemma is clear from Lemmas \ref{lem2.6} and \ref{lem3.1}:

\begin{lemma}\label{lem3.4}
Let $p$ be a point in $X \setminus (\partial X \cup \Focal (\partial X))$, 
and $\delta_{1}(p)$ the number $\delta_{1}$ guaranteed in Lemma $\ref{lem2.6}$ to the point $r :=p$. 
Then, for any $q \in X$ with $d (p, q) < \delta_{1}(p)$, there exists an open triangle 
${\rm OT}(\partial \wt{X}_{\lambda}, \tilde{p}, \tilde{q})$ in $\wt{X}_{\lambda}$ 
corresponding to the triangle ${\rm OT}(\partial X, p, q)$ in $X$ such that
\begin{equation}\label{lem3.4-length}
d(\partial \wt{X}_{\lambda},\tilde{p}) = d(\partial X, p), \quad 
d(\tilde{p},\tilde{q}) = d(p, q), \quad 
d(\partial \wt{X}_{\lambda},\tilde{q}) = d(\partial X, q)
\end{equation}
and that
\begin{equation}\label{lem3.4-angle}
\angle\,p \ge \angle\,\tilde{p}, \quad  
\angle\,q \ge \angle\,\tilde{q}.
\end{equation}
\end{lemma}

\bigskip

By Lemmas \ref{lem3.2} and \ref{lem3.4}, we have the following lemma. 

\begin{lemma}\label{lem3.5}
For every thin open triangle ${\rm OT}(\partial X, p, q)$ in $X$ with 
$\gamma \cap \Focal (\partial X) = \emptyset$, 
there exists an open triangle 
${\rm OT}(\partial \wt{X}_{\lambda}, \tilde{p}, \tilde{q})$ in $\wt{X}_{\lambda}$ such that
\begin{equation}\label{lem3.5-length}
d(\partial \wt{X}_{\lambda},\tilde{p}) = d(\partial X, p), \quad 
d(\tilde{p},\tilde{q}) = d(p, q), \quad 
d(\partial \wt{X}_{\lambda},\tilde{q}) = d(\partial X, q)
\end{equation}
and that
\begin{equation}\label{lem3.5-angle}
\angle\,p \ge \angle\,\tilde{p}, \quad  
\angle\,q \ge \angle\,\tilde{q}.
\end{equation}
Here $\gamma$ denotes the opposite side of ${\rm OT}(\partial X, p, q)$ 
to $\partial X$ emanating from $p$ to $q$.
\end{lemma}

\begin{proof}
It is sufficient to prove that 
\begin{equation}\label{lem3.5-1}
\max S = d(p, q),
\end{equation}
where 
$S$ denotes 
the set of all $s \in [0, d(p, q)]$ such that there exists an open triangle 
${\rm OT}(\partial \wt{X}_{\lambda}, \tilde{p}, \tilde{\gamma}(s)) \subset \wt{X}_{\lambda}$ 
corresponding to the triangle ${\rm OT}(\partial X, p, \gamma(s)) \subset X$ 
satisfying (\ref{lem3.5-length}) and (\ref{lem3.5-angle}) for $q = \gamma(s)$. 
From Lemma \ref{lem3.4}, it is clear that $S$ is non-empty. 
Supposing that $s_{0} := \max S < d(p, q)$, we will get a contradiction. 
Since $s_{0} \in S$, there exists an open triangle 
${\rm OT}(\partial \wt{X}_{\lambda}, \tilde{p}_{1}, \tilde{q}_{1}) \subset \wt{X}_{\lambda}$ 
corresponding to ${\rm OT}(\partial X, p, \gamma(s_{0})) \subset X$ 
such that (\ref{lem3.5-length}) and (\ref{lem3.5-angle}) 
hold for $q = \gamma(s_{0})$, $\tilde{p} = \tilde{p}_{1}$, and $\tilde{q} = \tilde{q}_{1}$. 
In particular, 
\begin{equation}\label{lem3.5-2}
\angle\,p \ge \angle\,\tilde{p}_{1}, \quad \angle\,(\partial X, \gamma (s_{0}), p) \ge \angle\, \tilde{q}_{1},
\end{equation}
where $\angle(\partial X, \gamma (s_{0}), p) $ denotes 
the angle between two sides joining $\gamma (s_{0})$ to $\partial X$ and 
$p$ forming the triangle ${\rm OT}(\partial X, p, \gamma(s_{0}))$. 
Let $\delta(\gamma (s_{0}))$ denote 
the number $\delta_{1}$ guaranteed in Lemma \ref{lem2.6} 
to the point $r :=\gamma (s_{0})$. Choose a sufficiently small number $\ve_{1}$ with 
\[
0 < \ve_{1} < \min \left\{ \delta (\gamma (s_{0})), d(p, q) - s_{0} \right\},
\]
and fix it.
By Lemma \ref{lem3.4}, we have an open triangle 
${\rm OT}(\partial \wt{X}_{\lambda}, \tilde{p}_{2}, \tilde{q}_{2}) \subset \wt{X}_{\lambda}$ 
corresponding to ${\rm OT}(\partial X, \gamma(s_{0}), \gamma(s_{0} + \ve_{1})) \subset X$ 
such that (\ref{lem3.5-length}) and (\ref{lem3.5-angle}) hold 
for $p = \gamma(s_{0})$, $q = \gamma(s_{0} + \ve_{1})$, $\tilde{p} = \tilde{p}_{2}$, 
and $\tilde{q} = \tilde{q}_{2}$. 
In particular, 
\begin{equation}\label{lem3.5-4}
\angle\,(\partial X, \gamma (s_{0}), \gamma (s_{0} + \ve_{1})) \ge \angle\,\tilde{p}_{2}, \quad 
\angle\,\gamma (s_{0} + \ve_{1}) \ge \angle\, \tilde{q}_{2}.
\end{equation}
Since 
\[
\angle\,(\partial X, \gamma (s_{0}), p) 
+ \angle\,(\partial X, \gamma (s_{0}), \gamma (s_{0} + \ve_{1})) = \pi,
\]
we get, by (\ref{lem3.5-2}) and (\ref{lem3.5-4}), 
\[
\angle\,\tilde{q}_{1} + \angle\,\tilde{p}_{2} \le \pi.
\]
Since ${\rm OT}(\partial X, p, q)$ is a thin open triangle, 
\[
\min \{ \inj (\tilde{p}_{1}), \inj (\tilde{q}_{2}) \} 
\ge 
L(\gamma) 
\ge 
d(\tilde{p}_{1}, \tilde{q}_{1}) + d(\tilde{p}_{2}, \tilde{q}_{2})
\]
holds. 
Thus, if we apply Lemma \ref{lem3.2} twice for the pair 
${\rm OT}(\partial \wt{X}_{\lambda}, \tilde{p}_{1}, \tilde{q}_{1})$ 
and ${\rm OT}(\partial \wt{X}_{\lambda}, \tilde{p}_{2}, \tilde{q}_{2})$, 
we get two open triangles 
${\rm OT}(\partial \wt{X}_{\lambda}, \wh{p}, \wh{\gamma}(s_{0} + \ve_{1}))$ and 
${\rm OT}(\partial \wt{X}_{\lambda}, \tilde{p}, \tilde{\gamma}(s_{0} + \ve_{1}))$ in $\wt{X}_{\lambda}$ 
such that 
\begin{equation}\label{lem3.5-5}
\angle\,\tilde{p}_{1} \ge \angle\,\wh{p}, \quad 
\angle\,\tilde{q}_{2} \ge \angle\, \tilde{\gamma}(s_{0} + \ve_{1}).
\end{equation}
Since both triangles ${\rm OT}(\partial \wt{X}_{\lambda}, \wh{p}, \wh{\gamma}(s_{0} + \ve_{1}))$ 
and ${\rm OT}(\partial \wt{X}_{\lambda}, \tilde{p}, \tilde{\gamma}(s_{0} + \ve_{1}))$ are isometric, 
we obtain 
\begin{equation}\label{lem3.5-6}
\angle\,\wh{p} = \angle\,\tilde{p}.
\end{equation}
Hence, by (\ref{lem3.5-2}), (\ref{lem3.5-4}), (\ref{lem3.5-5}), and (\ref{lem3.5-6}), 
both open triangles ${\rm OT}(\partial X, p, \gamma (s_{0} + \ve_{1}))$ 
and ${\rm OT}(\partial \wt{X}_{\lambda}, \tilde{p}, \tilde{\gamma}(s_{0} + \ve_{1}))$ satisfy 
(\ref{lem3.5-length}) for $q = \gamma(s_{0} + \ve_{1})$ and 
$\tilde{q} = \tilde{\gamma}(s_{0} + \ve_{1})$ and 
\[
\angle\,p \ge \angle\,\tilde{p}, \quad 
\angle\,\gamma (s_{0} + \ve_{1}) \ge \angle\,\tilde{\gamma}(s_{0} + \ve_{1}).
\] 
This implies that $s_{0} + \ve_{1}$ is in $S$. 
This contradicts the fact that $s_{0}$ is the maximum of $S$. 
Hence (\ref{lem3.5-1}) holds.
$\qedd$
\end{proof}

\begin{lemma}\label{lem3.6}
For every thin open triangle ${\rm OT}(\partial X, p, q)$ in $X$ with 
\begin{equation}\label{lem3.6-assume}
p \not\in \Focal(\partial X),
\end{equation}
there exists an open triangle 
${\rm OT}(\partial \wt{X}_{\lambda}, \tilde{p}, \tilde{q})$ in $\wt{X}_{\lambda}$ such that
\begin{equation}\label{lem3.6-length}
d(\partial \wt{X}_{\lambda},\tilde{p}) = d(\partial X, p), \quad 
d(\tilde{p},\tilde{q}) = d(p, q), \quad 
d(\partial \wt{X}_{\lambda},\tilde{q}) = d(\partial X, q)
\end{equation}
and that
\begin{equation}\label{lem3.6-angle}
\angle\,p \ge \angle\,\tilde{p}, \quad  
\angle\,q \ge \angle\,\tilde{q}.
\end{equation}
\end{lemma}

\begin{proof}
Let ${\rm OT}(\partial X, p, q) := (\partial X, p, q\,;\,\gamma, \mu_{1}, \mu_{2})$ 
be a thin open triangle in $X$ satisfying (\ref{lem3.6-assume}), and we fix it. 
Since $p$ is not a focal point of $\partial X$, and $q$ is not a cut point of $p$, 
it follows from Lemma \ref{lem2.2} that 
there exists a sequence 
$
\left\{
\gamma_{i} : [0, l
_{i}] \rightarrow
 X
\right\}_{i \in \N}
$ 
of minimal geodesic segments $\gamma_{i}$ emanating from $p = \gamma_{i}(0)$ 
convergent to the opposite side $\gamma$ of ${\rm OT}(\partial X, p, q)$ to $\partial X$ such that 
\[
\gamma_{i}([0, l
_{i}]) \cap \FC(\partial X) = \emptyset,
\]
and that 
\[
\lim_{i \to \infty} \frac{1}{\| \exp_{q}^{-1}(\gamma_{i}(l
_{i})) \|} \exp_{q}^{-1}(\gamma_{i}(l
_{i})) 
= - \gamma' (l
),
\]
where $l:= d(p, q)$. 
Then, we may find a sufficiently large $i_{0} \in \N$ 
such that 
an open triangle ${\rm OT}(\partial X, p, \gamma_{i}(l_{i}))
=(\partial X, p, \gamma_{i}(l_{i})\,;\,\gamma_{i}, \mu_{1}, \eta_{i})$ 
is thin in $X$ for each $i \ge i_{0}$. 
Here each $\eta_{i}$ is a $\partial X$-segment to $\gamma_{i}(l
_{i})$.
Choose any $i \ge i_{0}$ and fix it. 
By Lemma \ref{lem3.5}, 
there exists an open triangle 
$
{\rm OT}
(\partial \wt{X}_{\lambda}, \tilde{p}, \tilde{\gamma}_{i}(l_{i})) 
= (\partial \wt{X}_{\lambda}, \tilde{p}, \tilde{\gamma}_{i}(l_{i})\,;\,
\tilde{\gamma}_{i}, \tilde{\mu}_{1}, \tilde{\eta}_{i}) \subset \wt{X}_{\lambda}
$ 
corresponding to ${\rm OT}(\partial X, p, \gamma_{i}(l_{i}))$ such that (\ref{lem3.5-length}) 
hold for $q = \gamma_{i}(l_{i})$, and 
\begin{equation}\label{lem3.6-1}
\angle(
- \mu'_{1}(d(\partial X, p)), \gamma'_{i}(0)
) 
\ge 
\angle(
- \tilde{\mu}'_{1}(d(\partial X, p)), \tilde{\gamma}'_{i}(0)
), 
\end{equation}
\begin{equation}\label{lem3.6-2}
\angle(
\eta'_{i}(d(\partial X, \gamma_{i}(l_{i}))), \gamma'_{i}(l_{i})
) 
\ge 
\angle(
\tilde{\eta}'_{i}(d(\partial X, \gamma_{i}(l_{i}))), \tilde{\gamma}'_{i}(l_{i})
).
\end{equation}
Since $\lim_{i \to \infty} \gamma'_{i}(0) = \gamma'(0)$, 
\begin{equation}\label{lem3.6-3}
\angle\,p 
= 
\lim_{i \to \infty}
\angle(
- \mu'_{1}(d(\partial X, p)), \gamma'_{i}(0)
).
\end{equation}
On the other hand, 
\begin{equation}\label{lem3.6-4}
\angle\,q \ge 
\limsup_{i \to \infty}\angle(
\eta'_{i}(d(\partial X, \gamma_{i}(l
_{i}))), \gamma'_{i}(l
_{i})
)
\end{equation}
holds by \cite[Lemma 2.1]{IT3}. 
Then, from (\ref{lem3.6-1}), (\ref{lem3.6-2}), (\ref{lem3.6-3}), (\ref{lem3.6-4}), it follows that 
\[
\angle\,p 
\ge 
\lim_{i \to \infty} 
\angle(
- \tilde{\mu}'_{1}(d(\partial X, p)), \tilde{\gamma}'_{i}(0)
),
\]
and that 
\[
\angle\,q 
\ge 
\lim_{i \to \infty} 
\angle(
\tilde{\eta}'_{i}(d(\partial X, \gamma_{i}(l_{i}))), \tilde{\gamma}'_{i}(l_{i})
).
\]
By taking the limit of the sequence 
${\rm OT}(\partial \wt{X}_{\lambda}, \tilde{p}, \tilde{\gamma}_{i}(l_{i}))$, 
we therefore get an open triangle 
${\rm OT}(\partial \wt{X}_{\lambda}, \tilde{p}, \tilde{q})$ 
corresponding to ${\rm OT}(\partial X, p, q) \subset X$ such that  
(\ref{lem3.6-length}) and (\ref{lem3.6-angle}) hold.
$\qedd$
\end{proof}

\begin{lemma}\label{lem3.7}
For every thin open triangle ${\rm OT}(\partial X, p, q)$ in $X$, 
there exists an open triangle ${\rm OT}(\partial \wt{X}_{\lambda}, \tilde{p}, \tilde{q})$ 
in $\wt{X}_{\lambda}$ such that
\begin{equation}\label{lem3.7-length}
d(\partial \wt{X}_{\lambda},\tilde{p}) = d(\partial X, p), \quad 
d(\tilde{p},\tilde{q}) = d(p, q), \quad 
d(\partial \wt{X}_{\lambda},\tilde{q}) = d(\partial X, q)
\end{equation}
and that
\begin{equation}\label{lem3.7-angle}
\angle\,p \ge \angle\,\tilde{p}, \quad  
\angle\,q \ge \angle\,\tilde{q}.
\end{equation}
\end{lemma}

\begin{proof}
Let ${\rm OT}(\partial X, p, q) := (\partial X, p, q\,;\,\gamma, \mu_{1}, \mu_{2})$ 
be a thin open triangle in $X$, and we fix it. 
Take any sufficiently small $\ve > 0$ such that 
$q$ is not in $\Cut (p_{\ve})$, 
where we set $p_{\ve} := \mu_{1}(d(\partial X, p) - \ve)$. 
Let $\mu_{\ve}$ denote the restriction of $\mu_{1}$, i.e., 
$\mu_{\ve}(t) := \mu_{1}(t)$ on $[0, d(\partial X, p) - \ve]$. 
Without loss of generality, 
we may assume that the open triangle 
${\rm OT}(\partial X, p_{\ve}, q) = (\partial X, p_{\ve}, q\,;\,\gamma_{\ve}, \mu_{\ve}, \mu_{2})$ is thin. 
Here $\gamma_{\ve} : [0, l
_{\ve}] \rightarrow
 X$ denotes the minimal geodesic segment 
emanating from $p_{\ve} = \gamma_{\ve}(0)$ to $q = \gamma_{\ve}(l
_{\ve})$. 
Since $p_{\ve} \not\in \Focal (\partial X)$, 
it follows from Lemma \ref{lem3.6} that 
there exists an open triangle 
$
{\rm OT}(\partial \wt{X}_{\lambda}, \tilde{p}_{\ve}, \tilde{q}) 
= (\partial \wt{X}_{\lambda}, \tilde{p}_{\ve}, \tilde{q}\,;\,\tilde{\gamma}_{\ve}, \tilde{\mu}_{\ve}, \tilde{\mu}_{2}) 
\subset \wt{X}_{\lambda}
$ 
corresponding to ${\rm OT}(\partial X, p_{\ve}, q)$ such that (\ref{lem3.6-length})
holds for $p = p_{\ve}$, and that 
\[
\angle
(
- \mu'_{\ve}(d(\partial X, p_{\ve})), \gamma'_{\ve}(0)
) 
\ge 
\angle
(
- \tilde{\mu}'_{\ve}(d(\partial X, p_{\ve})), \tilde{\gamma}'_{\ve}(0)
), 
\]
\[
\angle
(
\mu'_{2}(d(\partial X, q)), \gamma'_{\ve}(l
_{\ve})
) 
\ge 
\angle
(
\tilde{\mu}'_{2}(d(\partial X, q)), \tilde{\gamma}'_{\ve}(l
_{\ve})
)_.
\]
Since $\lim_{\ve \downarrow 0} \gamma_{\ve} = \gamma$,  
we have 
\[
\angle\,p 
=
\lim_{\ve \downarrow 0}\angle
(
- \mu'_{\ve}(d(\partial X, p_{\ve})), \gamma'_{\ve}(0)
)
\]
and 
\[
\angle\,q 
= 
\lim_{\ve \downarrow 0}
\angle
(
\mu'_{2}(d(\partial X, q)), \gamma'_{\ve}(l
_{\ve})
).
\]
If $\ve$ goes to $0$, 
we therefore get an open triangle 
${\rm OT}(\partial \wt{X}_{\lambda}, \tilde{p}, \tilde{q}) \subset \wt{X}_{\lambda}$ corresponding to 
${\rm OT}(\partial X, p, q) \subset X$ such that  (\ref{lem3.7-length}) and (\ref{lem3.7-angle}) hold.
$\qedd$
\end{proof}

\bigskip

By taking the limit of $\lambda$, it follows from Lemma \ref{lem3.7} that 
we have the lemma on thin open triangles:

\begin{lemma}{\bf (Lemma on thin open triangles)}\label{lem3.8} 
For every thin open triangle ${\rm OT}(\partial X, p, q)$ in $X$, 
there exists an open triangle 
${\rm OT}(\partial \wt{X}, \tilde{p}, \tilde{q})$ in $\wt{X}$ such that
\begin{equation}\label{lem3.8-length}
d(\partial \wt{X},\tilde{p}) = d(\partial X, p), \quad 
d(\tilde{p},\tilde{q}) = d(p, q), \quad 
d(\partial \wt{X},\tilde{q}) = d(\partial X, q)
\end{equation}
and that
\begin{equation}\label{lem3.8-angle}
\angle\,p \ge \angle\,\tilde{p}, \quad  
\angle\,q \ge \angle\,\tilde{q}.
\end{equation}
\end{lemma}

\begin{remark} 
From Section \ref{sec:LGV} and this section, 
it has been clarified that we can prove Lemma on thin open triangles 
by the second variation formula and the Warner comparison theorem. 
But, we can not do by the first variation formula and the Berger comparison theorem.
\end{remark}

\section{The opposite side to $\partial X$ of an open triangle}\label{sec:Rem}

In Definition \ref{def3.3}, 
the opposite side to $\partial X$ of a thin open triangle 
is defined not to meet the boundary. 
In this section, 
we will show that the opposite side to $\partial X$ of 
any open triangle on {\bf any} complete connected Riemannian manifold $X$ 
with smooth boundary $\partial X$ does not meet $\partial X$, if $\partial X$ is convex. 

\begin{lemma}\label{lem3.10}
Let $(X, \partial X)$ be a complete connected Riemannian 
$n$-dimensional manifold $X$ with smooth boundary $\partial X$ 
whose radial curvature is bounded from below by that of 
a model surface $(\wt{X}, \partial \wt{X})$. 
If $\partial X$ is convex, then, 
for any open triangle ${\rm OT}(\partial X, p, q)$ in $X$, 
the opposite side $\gamma$ of ${\rm OT}(\partial X, p, q)$ to $\partial X$ 
emanating from $p$ to $q$ does not meet $\partial X$.
\end{lemma}

\begin{proof}
Suppose that $\gamma$ intersects $\partial X$ at $\gamma (s_{0})$ 
for some $s_{0} \in (0, d(p, q))$. 
Without loss of generality, 
we may assume that 
\[
\gamma ((0, s_{0})) \cap \partial X = \emptyset.
\]
Since $\gamma$ intersects $\partial X$ at $\gamma (s_{0})$, 
$\gamma$ is tangent to $\partial X$ at $\gamma (s_{0})$.\par 
It is well-known that each point of $\wt{X}$ admits a normal convex neighborhood. 
Hence, there exists a constant $C_{0} > 0$ such that 
$\inj(\tilde{q}_{s}) > C_{0}$ for all $s \in [0, s_{0}]$, 
where $\tilde{q}_{s}$ denotes a point in $\wt{X}$ satisfying 
\[
d(\partial \wt{X}, \tilde{q}_{s}) = d(\partial X, \gamma (s)).
\] 
By this property, we may choose a number $s_{1} \in (0, s_{0})$ in such a way that 
\[
L(\gamma|_{[s_{1},\,s_{0}]}) = s_{0} - s_{1} < \inj(\tilde{q}_{s})
\]
for all $s \in [s_{1}, s_{0}]$. 
Therefore, for each $s_{2} \in [s_{1}, s_{0})$, 
any open triangle ${\rm OT}(\partial X, \gamma(s_{2}), \gamma(s_{3}))$ in $X$ is thin, 
if $s_{3} \in [s_{1}, s_{0}) \setminus \{ s_{2} \}$ is sufficiently close to $s_{2}$.\par
Let $S$ denote the set of all $s \in (s_{1}, s_{0})$ such that 
there exists an open triangle 
${\rm OT}(\partial \wt{X}, \tilde{\gamma}(s_{1}), \tilde{\gamma}(s)) \subset \wt{X}$ 
corresponding to the triangle ${\rm OT}(\partial X, \gamma(s_{1}), \gamma(s)) \subset X$ 
satisfying 
\begin{equation}\label{lem3.10-1}
d(\partial \wt{X},\tilde{\gamma}(s_{1})) = d(\partial X, \gamma(s_{1})), \ 
d(\tilde{\gamma}(s_{1}), \tilde{\gamma}(s)) = s - s_{1}, \ 
d(\partial \wt{X},\tilde{\gamma}(s)) = d(\partial X, \gamma(s)),
\end{equation}
and that
\begin{equation}\label{lem3.10-2}
\angle\,\gamma(s_{1}) \ge \angle\,\tilde{\gamma}(s_{1}), \quad  
\angle\,\gamma(s) \ge \angle\,\tilde{\gamma}(s).
\end{equation}
By Lemma \ref{lem3.8}, the set $S$ is non-empty. 
By the similar argument in the proof of Lemma \ref{lem3.5}, 
we see that $\sup S = s_{0}$. 
Hence, there exists a decreasing sequence $\{\ve_{i} \}_{i \in \N}$ convergent to $0$ such that 
$s_{0} - \ve_{i}$ is in $S$ for all $i \in \N$. 
For each $i \in \N$, 
there exists an open triangle 
${\rm OT}(\partial \wt{X}, \tilde{\gamma}(s_{1}), \tilde{\gamma}(s_{0} - \ve_{i})) \subset \wt{X}$ 
corresponding to the triangle ${\rm OT}(\partial X, \gamma(s_{1}), \gamma(s_{0} - \ve_{i})) \subset X$ 
such that (\ref{lem3.10-1}) and (\ref{lem3.10-2}) hold for $\gamma(s) = \gamma(s_{0} - \ve_{i})$.
Since 
$\gamma$ is tangent to $\partial X$ at $\gamma (s_{0})$, 
we have 
\[\lim_{s \uparrow s_{0}} \angle\,\gamma(s) = \frac{\pi}{2}.
\]
Thus, by (\ref{lem3.10-2}) for $\gamma(s) = \gamma(s_{0} - \ve_{i})$, we get 
\[
\limsup_{i \to \infty} \angle\,\tilde{\gamma}(s_{0} - \ve_{i}) \le \frac{\pi}{2}.
\]
If 
\[
\liminf_{i \to \infty} \angle\,\tilde{\gamma}(s_{0} - \ve_{i}) < \frac{\pi}{2}
\]
holds, 
the opposite side to $\partial \wt{X}$ of 
${\rm OT}(\partial \wt{X}, \tilde{\gamma}(s_{1}), \tilde{\gamma}(s_{0} - \ve_{i}))$ 
meets $\partial \wt{X}$ for some $s_{2} \in (s_{1}, s_{0})$ sufficiently close to $s_{0}$, 
which contradicts the fact that $\partial \wt{X}$ is totally geodesic. 
Hence, 
\[
\lim_{i \to \infty} \angle\,\tilde{\gamma}(s_{0} - \ve_{i}) = \frac{\pi}{2}
\]
holds. 
Thus, the opposite side to $\partial \wt{X}$ of the limit open triangle 
${\rm OT}(\partial \wt{X}, \tilde{\gamma}(s_{1}), \tilde{\gamma}(s_{0} - \ve_{i}))$ as $i \to \infty$ 
is tangent to $\partial \wt{X}$. 
This is also a contradiction, 
since $\partial \wt{X}$ is totally geodesic. 
Therefore, $\gamma$ does not meet $\partial X$.
$\qedd$
\end{proof}

By the same argument as in the proof of \cite[Lemma 5.1]{KT2}, 
we have the following lemma. 

\begin{lemma}\label{lem2009-01-13}
For any complete connected Riemannian manifold $X$ with smooth boundary $\partial X$ 
and any compact subset $K$ of $\partial X$, 
there exists a locally Lipschitz function $G(t)$ on $[0, \infty)$ 
such that the radial curvature of $X$ with respect to any $\partial X$-segment 
emanating from $K$ is bounded from below by that of the model surface with 
radial curvature function $G(t)$.
\end{lemma}

\begin{proposition}\label{prop2009-01-13}
Let $X$ be a complete connected Riemannian manifold $X$ with smooth boundary $\partial X$. 
If $\partial X$ is convex, 
then the opposite side to $\partial X$ of any open triangle on $X$ 
does not meet $\partial X$. 
\end{proposition}

\begin{proof}
By the definition of an open triangle, 
both feet $p$ and $q$ of an open triangle ${\rm OT}(\partial X, p, q) \subset X$ 
are contained in a bounded set of $\partial X$ . Hence 
it is clear from Lemmas \ref{lem3.10} and \ref{lem2009-01-13} that 
$\gamma$ does not meet $\partial X$.
$\qedd$
\end{proof}

\section{Alexandrov's convexity}\label{sec:AC}
Our purpose of this section is to establish the Alexandrov convexity (Lemma \ref{lem4.4}). 

\bigskip

Throughout this section, let $(\wt{X}, \partial \wt{X})$ denote 
a model surface with its metric (\ref{model-metric}). 
It follows from Lemma \ref{lem4.1} that 
\begin{equation}\label{lem4.1-0}
\lim_{s \downarrow s_{0}} 
\frac{
d(\tilde{p}, \tilde{\tau}_{c}(s)) - d(\tilde{p}, \tilde{\tau}_{c}(s_{0}))
}
{
s - s_{0}
}
\ge 0
\end{equation}
for each $s_{0} > \tilde{y}(\tilde{p})$. The following two lemmas are useful for proving 
that the function $\cD$ defined in Lemma \ref{lem4.4} is locally Lipschitz. 
In the first lemma, we will prove that the left-hand term of the equation (\ref{lem4.1-0}) 
is {\bf strictly} positive\,:

\begin{lemma}\label{lem4.2}
For each $\partial \wt{X}$-ray $\tilde{\mu} : [0, \infty) \rightarrow \wt{X}$ and each number 
$a_{0}, c_{0} > 0$, $s_{0} > \tilde{y}(\tilde{\mu}(0))$, 
there exist numbers $\ve_{1} > 0$ and $\delta > 0$ such that 
\begin{equation}\label{lem4.2-0}
|d(\tilde{\mu}(a), \tilde{\tau}_{c}(s)) - d(\tilde{\mu}(a), \tilde{\tau}_{c}(s_{0}))| 
\ge 
|s -s_{0}| \cdot m(c) \cdot \sin \ve_{1}
\end{equation}
holds for all $a \in (a_{0} - \delta, a_{0} + \delta)$, 
$c \in (c_{0} - \delta, c_{0} + \delta)$, and 
$s \in (s_{0} - \delta, s_{0} + \delta)$.
\end{lemma}

\begin{proof}
We choose a positive number $\delta$ less than 
$\min \{ a_{0}, c_{0},  s_{0} - \tilde{y}(\tilde{\mu}(0))\}$, 
and fix it. Let $a, c, s$ be any numbers in 
$(a_{0} - \delta, a_{0} + \delta)$, 
$(c_{0} - \delta, c_{0} + \delta)$, and 
$(s_{0} - \delta, s_{0} + \delta)$, respectively. 
Since no minimal geodesic segment joining $\tilde{\mu}(a)$ to $\tilde{\tau}_{c}(s)$ is perpendicular 
to $\tilde{\tau}_{c}$, there exists a positive number 
$\ve_{1} \in (0, \pi / 2)$ such that \begin{equation}\label{lem4.2-1}
\Phi (\tilde{\gamma}, s) := 
\angle (\tilde{\tau}'_{c}(s), \tilde{\gamma}'(d(\tilde{\mu}(a), \tilde{\tau}_{c}(s)))) 
\le \frac{\pi}{2} - \ve_{1}
\end{equation}
holds for all 
$a \in (a_{0} - \delta, a_{0} + \delta)$
$c \in (c_{0} - \delta, c_{0} + \delta)$, 
$s \in (s_{0} - \delta, s_{0} + \delta)$, and 
minimal geodesic segments $\tilde{\gamma}$ joining $\tilde{\mu}(a)$ to $\tilde{\tau}_{c}(s)$. 
By \cite[Lemma 2.1]{IT3} and (\ref{lem4.2-1}), we have  
\[
\liminf_{s \downarrow s_{0}}
\frac{
d(\tilde{\mu}(a), \tilde{\tau}_{c}(s)) - d(\tilde{\mu}(a), \tilde{\tau}_{c}(s_{0}))
}
{
d(\tilde{\tau}_{c}(s), \tilde{\tau}_{c}(s_{0}))
}
\ge 
\sin \ve_{1}.
\]
This equation implies (\ref{lem4.2-0}) (see the proof of \cite[Lemma 4.2]{KT2} 
for the detail of this proof).$\qedd$
\end{proof}

\bigskip

Notice that, for given a triple $(a, b, c)$ of positive numbers $a, b, c$, 
there exists an open triangle 
${\rm OT}(\partial \wt{X}, \tilde{p}, \tilde{q})$ in $\wt{X}$ 
with $\angle\,\tilde{p}, \angle\,\tilde{q} \in (0, \pi)$ satisfying 
$a = d(\partial \wt{X},\tilde{p})$, 
$b = d(\tilde{p},\tilde{q})$, and $c = d(\partial \wt{X},\tilde{q})$ 
if and only if 
\[
(a, b, c) \in T := \{(a, b, c) \in \R^{3}\,;\, a, c > 0, |a - c| < b \}.
\]
By Lemma \ref{lem4.1}, 
the existence of such a triangle ${\rm OT}(\partial \wt{X}, \tilde{p}, \tilde{q}) 
= (\partial \wt{X}, \tilde{p}, \tilde{q}\,;\,\tilde{\gamma}, \tilde{\mu}_{1}, \tilde{\mu}_{2})$ 
is unique up to an isometry except for the opposite side $\tilde{\gamma}$ to $\partial \wt{X}$. 
Hence, 
\[
\Theta (a, b, c) := |\tilde{y}( \tilde{\mu}_{1}(0) ) - \tilde{y}(\tilde{\mu}_{2} (0) ) |
\]
is a well-defined function on the set $T$. 

\begin{lemma}\label{lem4.3}
The function $\Theta (a, b, c)$ is locally Lipschitz.
\end{lemma}
 
\begin{proof}
Choose any point $(a_{0}, b_{0}, c_{0}) \in T$, and fix it. 
Let $\tilde{\mu}_{1} : [0, \infty) \rightarrow
 \wt{X}$ be the $\partial \wt{X}$-ray with 
$\tilde{y} (\tilde{\mu}_{1}(0)) = 0$. 
Moreover, we choose the $\partial \wt{X}$-ray $\tilde{\mu}_{2} : [0, \infty) \rightarrow
 \wt{X}$ 
in such a way that $d(\tilde{\mu}_{1}(a_{0}), \tilde{\mu}_{2}(c_{0})) = b_{0}$ 
and $\tilde{y} (\tilde{\mu}_{2}(0)) > 0$. 
By setting $\tilde{p}_{0} := \tilde{\mu}_{1}(a_{0})$ and $\tilde{q}_{0} := \tilde{\mu}_{2}(c_{0})$, 
we hence get an open triangle 
${\rm OT}(\partial \wt{X}, \tilde{p}_{0}, \tilde{q}_{0}) \subset \wt{X}$ 
with $\angle\,\tilde{p}_{0}, \angle\,\tilde{q}_{0} \in (0, \pi)$ satisfying 
\[
a_{0} = d(\partial \wt{X},\tilde{p}_{0}), \quad 
b_{0} = d(\tilde{p}_{0}, \tilde{q}_{0}), \quad 
c_{0} = d(\partial \wt{X},\tilde{q}_{0}).
\]
First we will prove that 
\begin{equation}\label{lem4.3-1}
|\Theta (a_{0} + \Delta a, b_{0}, c_{0}) - \Theta (a_{0}, b_{0}, c_{0}) | 
\le 
\frac{1}{m(c_{0}) \sin \ve_{1}} |\Delta a| 
\end{equation}
for all $\Delta a \in \R$ with $|\Delta a| < \delta$. 
Here the numbers $\ve_{1}$ and $\delta$ are the constants guaranteed to $a_{0}$, $c_{0}$, 
and $s_{0} := \tilde{y} (\tilde{\mu}_{2}(0)) > 0$ in Lemma \ref{lem4.2}. 
Let $\tilde{\tau}_{c_{0}} : \R \rightarrow \wt{X}$ be the arc $\tilde{x} = c_{0}$. 
Then, we may find a point $\tilde{q}_{\Delta a}$ on $\tilde{\tau}_{c_{0}}$ satisfying 
$b_{0} = d(\tilde{p}_{\Delta a}, \tilde{q}_{\Delta a})$, 
where $\tilde{p}_{\Delta a} := \tilde{\mu}_{1}(a_{0} + \Delta a)$. 
Thus, we also get an open triangle 
${\rm OT}(\partial \wt{X}, \tilde{p}_{\Delta a}, \tilde{q}_{\Delta a}) \subset \wt{X}$ with 
$\angle\,\tilde{p}_{\Delta a}, \angle\,\tilde{q}_{\Delta a} \in (0, \pi)$ satisfying 
\[
a_{0} + \Delta a = d(\partial \wt{X},\tilde{p}_{\Delta a}), \quad 
b_{0} = d(\tilde{p}_{\Delta a}, \tilde{q}_{\Delta a}), \quad 
c_{0} = d(\partial \wt{X},\tilde{q}_{\Delta a}).
\]
Let $\tilde{\mu}_{\Delta a}$ denote the side 
of ${\rm OT}(\partial \wt{X}, \tilde{p}_{\Delta a}, \tilde{q}_{\Delta a})$ joining $\partial \wt{X}$ to 
$\tilde{q}_{\Delta a}$.  
By definition, 
\begin{equation}\label{lem4.3-2}
\Theta (a_{0} + \Delta a, b_{0}, c_{0}) 
= 
\tilde{y}(\tilde{\mu}_{\Delta a} (0)),
\end{equation}
and 
\begin{equation}\label{lem4.3-3}
\Theta (a_{0}, b_{0}, c_{0}) 
=
\tilde{y}(\tilde{\mu}_{2}(0))
= 
s_{0}.
\end{equation}
Here we may assume that $\tilde{y}(\tilde{\mu}_{\Delta a} (0)) > 0$. 
It is clear from (\ref{lem4.3-2}) and (\ref{lem4.3-3}) that 
\begin{equation}\label{lem4.3-4}
|\Delta_{a} \Theta | = | \tilde{y}(\tilde{\mu}_{\Delta a} (0)) - s_{0} |,
\end{equation}
where we set 
\[
\Delta_{a} \Theta := \Theta (a_{0} + \Delta a, b_{0}, c_{0}) - \Theta (a_{0}, b_{0}, c_{0}).
\]
Thus, by (\ref{lem4.3-4}), the length $|\Delta s|$ of the subarc of $\tilde{\tau}_{c_{0}}$ with end points 
$\tilde{q}_{\Delta a}$ and $\tilde{q}_{0}$ is equal to 
\begin{equation}\label{lem4.3-5}
|\Delta s| 
=
m(c_{0}) \cdot |\tilde{y}(\tilde{\mu}_{\Delta a} (0)) - s_{0} |
=
m(c_{0}) \cdot |\Delta_{a} \Theta |.
\end{equation}
It follows from Lemma \ref{lem4.2} and (\ref{lem4.3-5}) that 
\begin{equation}\label{lem4.3-6}
|d(\tilde{p}_{0}, \tilde{q}_{0}) - d(\tilde{p}_{0}, \tilde{q}_{\Delta a})| 
\ge 
|s_{0} - \tilde{y}(\tilde{\mu}_{\Delta a} (0))| \cdot m(c_{0}) \cdot \sin \ve_{1} 
= 
|\Delta s| \sin \ve_{1}. 
\end{equation}
Since 
\[
b_{0} = d(\tilde{p}_{0}, \tilde{q}_{0}) = d(\tilde{p}_{\Delta a}, \tilde{q}_{\Delta a}),
\]
we get, by (\ref{lem4.3-6}), 
\begin{equation}\label{lem4.3-7}
|d(\tilde{p}_{\Delta a}, \tilde{q}_{\Delta a}) - d(\tilde{p}_{0}, \tilde{q}_{\Delta a})| 
\ge 
|\Delta s| \sin \ve_{1}. 
\end{equation}
By the triangle inequality, we have
\begin{equation}\label{lem4.3-8}
|\Delta a| = d(\tilde{p}_{0}, \tilde{p}_{\Delta a}) 
\ge 
|d(\tilde{p}_{\Delta a}, \tilde{q}_{\Delta a}) - d(\tilde{p}_{0}, \tilde{q}_{\Delta a})|.
\end{equation}
By combining the equations (\ref{lem4.3-5}), (\ref{lem4.3-7}), and (\ref{lem4.3-8}), 
we obtain (\ref{lem4.3-1}). 
Since $\Theta (a, b, c) = \Theta (c, b, a)$ for all $(a, b, c) \in T$, 
it is clear that 
\begin{equation}\label{lem4.3-9}
|\Theta (a_{0}, b_{0}, c_{0} +  \Delta c) - \Theta (a_{0}, b_{0}, c_{0}) | 
\le 
\frac{1}{m(c_{0}) \sin \ve_{1}} |\Delta c| 
\end{equation}
for all $\Delta c \in \R$ with $|\Delta c| < \delta$. 
We omit the proof of the following inequality, 
since the proof is similar to that of (\ref{lem4.3-1})\,:
\begin{equation}\label{lem4.3-10}
|\Theta (a_{0}, b_{0} +  \Delta b, c_{0}) - \Theta (a_{0}, b_{0}, c_{0}) | 
\le 
\frac{1}{m(c_{0}) \sin \ve_{1}} |\Delta b| 
\end{equation}
for all $\Delta b \in \R$ with $|\Delta b| < \delta$. 
Therefore, the function $\Theta (a, b, c)$ is locally Lipschitz at $(a_{0}, b_{0}, c_{0}) \in T$ by 
(\ref{lem4.3-1}), (\ref{lem4.3-9}), and (\ref{lem4.3-10}).
$\qedd$
\end{proof}

\begin{lemma}\label{lem4.4}{\bf (Alexandrov's convexity for open triangles)} 
Let $(X, \partial X)$ be a complete connected 
Riemannian $n$-dimensional manifold $X$ with smooth convex boundary $\partial X$ 
whose radial curvature is bounded from below by that of $(\wt{X}, \partial \wt{X})$, and let 
${\rm OT}(\partial X, p, q) = (\partial X, p, q\,;\,\gamma, \mu_{1}, \mu_{2})$ be a non-degenerate open triangle in $X$, i.e., $\angle\,p, \angle\,q \in (0, \pi)$. 
Assume that, for each open triangle 
\[
{\rm OT}(\partial X, \mu_{1}(at), \mu_{2}(ct)) 
= 
(\partial X, \mu_{1}(at), \mu_{2}(ct)\,;\,\gamma_{t}, \mu_{1}|_{[0, \,at]}, \mu_{2}|_{[0, \,ct]}), \quad 
t \in (0, 1], 
\]
where $a = d(\partial X, p)$ and $c = d(\partial X, q)$, there exists a unique open triangle 
\[
{\rm OT}(\partial \wt{X}, \tilde{\mu}_{1}^{(t)}(at), \tilde{\mu}_{2}^{(t)}(ct)) 
= 
(\partial \wt{X}, \tilde{\mu}_{1}^{(t)}(at), \tilde{\mu}_{2}^{(t)}(ct)\,;\,
\tilde{\gamma}_{t}, \tilde{\mu}_{1}^{(t)}|_{[0, \,at]}, \tilde{\mu}_{2}^{(t)}|_{[0, \,ct]})
\]
up to an isometry  in $\wt{X}$ such that  
\begin{equation}\label{lem4.4-length1}
d(\partial \wt{X},\tilde{\mu}_{1}^{(t)}(at))=d(\partial X, \mu_{1}(at)), \quad 
d(\partial \wt{X},\tilde{\mu}_{2}^{(t)}(ct))=d(\partial X, \mu_{2}(ct)),
\end{equation}
\begin{equation}\label{lem4.4-length2}
d(\tilde{\mu}_{1}^{(t)}(at), \tilde{\mu}_{2}^{(t)}(ct))=d(\mu_{1}(at), \mu_{2} (ct)),  
\end{equation}
and that 
\begin{equation}\label{lem4.4-angle}
\angle\,\mu_{1}(at) \ge \angle\,\tilde{\mu}_{1}^{(t)}(at), \quad
\angle\,\mu_{2}(ct) \ge \angle\,\tilde{\mu}_{2}^{(t)}(ct).
\end{equation}
Then, the function 
\[
\cD (t) := | \tilde{y}(\tilde{\mu}_{1}^{(t)}(0)) - \tilde{y}(\tilde{\mu}_{2}^{(t)} (0)) |
\]
is locally Lipschitz on $(0, 1)$, 
and non-increasing on $(0, 1]$.
\end{lemma}

\begin{proof}
We will state the outline of the proof, 
since the proof is very similar to \cite[Lemma 4.4]{KT2}. 
If we define a Lipschitz function $\varphi$ on $[0, 1]$ by 
\[
\varphi (t) := d(\mu_{1} (at), \mu_{2} (ct)),
\]
then the function $\cD(t)$ is equal to $\Theta (at, \varphi (t), ct)$. 
Hence $\cD (t)$ is locally Lipschitz on $(0, 1]$ by Lemma \ref{lem4.3}. 
From Dini's theorem \cite{D} (cf.\,\cite[Section 2.3]{H}, \cite[Theorem 7.29]{WZ}), 
the function $\cD (t)$ is differentiable for almost all $t \in (0, 1)$. 
Let $t_{0} \in (0, 1)$ be any number where $\cD (t)$ is differentiable. 
Then, by the assumption, we have 
an open triangle 
\[
{\rm OT}(\partial \wt{X}, \tilde{\mu}_{1}^{(t_{0})}(at_{0}), \tilde{\mu}_{2}^{(t_{0})}(ct_{0})) 
= 
(\partial \wt{X}, \tilde{\mu}_{1}^{(t_{0})}(at_{0}), \tilde{\mu}_{2}^{(t_{0})}(ct_{0})\,;\,
\tilde{\gamma}_{t_{0}}, \tilde{\mu}_{1}^{(t_{0})}|_{[0, \,at_{0}]}, \tilde{\mu}_{2}^{(t_{0})}|_{[0, \,ct_{0}]})
\]
in $\wt{X}$ corresponding to the triangle 
\[
{\rm OT}(\partial X, \mu_{1}(at_{0}), \mu_{2}(ct_{0})) 
= 
(\partial X, \mu_{1}(at_{0}), \mu_{2}(ct_{0})\,;\,
\gamma_{t_{0}}, \mu_{1}|_{[0, \,at_{0}]}, \mu_{2}|_{[0, \,ct_{0}]})
\]
in $X$ such that (\ref{lem4.4-length1}),  (\ref{lem4.4-length2}), and (\ref{lem4.4-angle}) hold. 
Since ${\rm OT}(\partial X, \mu_{1}(at_{0}), \mu_{2}(ct_{0})) $ is non-degenerate, 
we may assume, without loss of generality, that 
\[
0 = \tilde{y} (\tilde{\mu}_{1}^{(t_{0})}(at_{0})) < \tilde{y} (\tilde{\mu}_{2}^{(t_{0})}(ct_{0}) ). 
\]
Let $\tilde{\mu}, \tilde{\eta} : [0, \infty) \rightarrow
 \wt{X}$ be $\partial \wt{X}$-rays passing through 
$\tilde{\mu}_{1}^{(t_{0})}(at_{0}) = \tilde{\mu} (at_{0})$, 
$\tilde{\mu}_{2}^{(t_{0})}(ct_{0}) = \tilde{\eta}(ct_{0})$, respectively. 
We define a function 
\[
\wt{\psi}(t) := d(\tilde{\mu}(at), \tilde{\eta}(ct)).
\]
Since $\tilde{\gamma}_{t_{0}}$ is unique, 
we may prove that the function $\wt{\psi}(t)$
is differentiable at $t = t_{0}$, and that 
\begin{equation}\label{lem4.4-1}
\wt{\psi}'(t_{0}) 
= 
\cos (\angle\,\tilde{\mu}_{1}^{(t_{0})}(at_{0})) + \cos (\angle\, \tilde{\mu}_{2}^{(t_{0})}(ct_{0})).
\end{equation}
Indeed, let $\tilde{z}_{0}$ and  $\tilde{z}_{t}$ denote the midpoints of $\tilde{\gamma}_{t_{0}}$ 
and $\tilde{\mu}(at)\tilde{\eta}(ct)$, respectively. 
Here, $\tilde{\mu}(at)\tilde{\eta}(ct)$ denotes a minimal geodesic segment 
joining $\tilde{\mu}(at)$ to $\tilde{\eta}(ct)$.
Since there exists a unique minimal geodesic segment 
joining $\tilde{\mu}_{1}^{(t_{0})}(at_{0}) = \tilde{\mu} (at_{0})$ to 
$\tilde{\mu}_{2}^{(t_{0})}(ct_{0}) = \tilde{\eta}(ct_{0})$, 
\begin{equation}\label{lem4.4-09-01-14}
\lim_{t \to t_{0}} \tilde{z}_{t} = \tilde{z}_{0}
\end{equation}
holds. 
By the triangle inequality, we have 
\[
\wt{\psi}(t) - \wt{\psi}(t_{0}) \le d(\tilde{\mu}(at), \tilde{z}_{0}) + d(\tilde{\eta}(ct), \tilde{z}_{0}) 
- d(\tilde{\mu}(at_{0}), \tilde{z}_{0}) - d(\tilde{\eta}(ct_{0}), \tilde{z}_{0})
\]
and 
\[
\wt{\psi}(t) - \wt{\psi}(t_{0}) \ge d(\tilde{\mu}(at), \tilde{z}_{t}) + d(\tilde{\eta}(ct), \tilde{z}_{t}) 
- d(\tilde{\mu}(at_{0}), \tilde{z}_{t}) - d(\tilde{\eta}(ct_{0}), \tilde{z}_{t}).
\]
Hence, 
\begin{align}\label{lem4.4-1.2new}
&\limsup_{t \downarrow t_{0}} \frac{\wt{\psi}(t) - \wt{\psi}(t_{0})}{ t - t_{0}}\\[2mm]
&\le 
\limsup_{t \downarrow t_{0}} 
\frac{d(\tilde{\mu}(at), \tilde{z}_{0}) - d(\tilde{\mu}(at_{0}), \tilde{z}_{0}) }{ t - t_{0} } 
+ 
\limsup_{t \downarrow t_{0}} 
\frac{d(\tilde{\eta}(ct), \tilde{z}_{0}) - d(\tilde{\eta}(ct_{0}), \tilde{z}_{0}) }{ t - t_{0} }\notag
\end{align}
and 
\begin{align}\label{lem4.4-1.3new}
&\liminf_{t \downarrow t_{0}} \frac{\wt{\psi}(t) - \wt{\psi}(t_{0})}{ t - t_{0}}\\[2mm]
&\ge 
\liminf_{t \downarrow t_{0}} 
\frac{d(\tilde{\mu}(at), \tilde{z}_{t}) - d(\tilde{\mu}(at_{0}), \tilde{z}_{t}) }{ t - t_{0} } 
+ 
\liminf_{t \downarrow t_{0}} 
\frac{d(\tilde{\eta}(ct), \tilde{z}_{t}) - d(\tilde{\eta}(ct_{0}), \tilde{z}_{t}) }{ t - t_{0} }\notag
\end{align}
hold. 
From the first variation formula, we have 
\begin{equation}\label{lem4.4-1.4new}
\limsup_{t \downarrow t_{0}} 
\frac{d(\tilde{\mu}(at), \tilde{z}_{0}) - d(\tilde{\mu}(at_{0}), \tilde{z}_{0}) }{ t - t_{0} } 
= \cos (\angle\,\tilde{\mu}_{1}^{(t_{0})}(at_{0}))
\end{equation}
and 
\begin{equation}\label{lem4.4-1.5new}
\limsup_{t \downarrow t_{0}} 
\frac{d(\tilde{\eta}(ct), \tilde{z}_{0}) - d(\tilde{\eta}(ct_{0}), \tilde{z}_{0}) }{ t - t_{0} } 
= \cos (\angle\, \tilde{\mu}_{2}^{(t_{0})}(ct_{0})).
\end{equation}
By imitating the proof of \cite[Lemma 2.1]{IT3}, we obtain 
\begin{equation}\label{lem4.4-1.6new}
\liminf_{t \downarrow t_{0}} 
\frac{d(\tilde{\mu}(at), \tilde{z}_{t}) - d(\tilde{\mu}(at_{0}), \tilde{z}_{t}) }{ t - t_{0} } 
=  \cos (\angle\,\tilde{\mu}_{1}^{(t_{0})}(at_{0}))
\end{equation}
and 
\begin{equation}\label{lem4.4-1.7new}
\liminf_{t \downarrow t_{0}} 
\frac{d(\tilde{\eta}(ct), \tilde{z}_{t}) - d(\tilde{\eta}(ct_{0}), \tilde{z}_{t}) }{ t - t_{0} } 
=  \cos (\angle\, \tilde{\mu}_{2}^{(t_{0})}(ct_{0})).
\end{equation}
In the above equations, notice (\ref{lem4.4-09-01-14}). 
Combining (\ref{lem4.4-1.2new}), (\ref{lem4.4-1.3new}), (\ref{lem4.4-1.4new}), 
(\ref{lem4.4-1.5new}), (\ref{lem4.4-1.6new}), and (\ref{lem4.4-1.7new}), 
we have 
\begin{equation}\label{lem4.4-1.8new}
\lim_{t \downarrow t_{0}} \frac{\wt{\psi}(t) - \wt{\psi}(t_{0})}{ t - t_{0}} 
= 
\cos (\angle\,\tilde{\mu}_{1}^{(t_{0})}(at_{0})) + \cos (\angle\, \tilde{\mu}_{2}^{(t_{0})}(ct_{0})).
\end{equation}
By the same way, we also see 
\begin{equation}\label{lem4.4-1.9new}
\lim_{t \uparrow t_{0}} \frac{\wt{\psi}(t) - \wt{\psi}(t_{0})}{ t - t_{0}} 
= 
\cos (\angle\,\tilde{\mu}_{1}^{(t_{0})}(at_{0})) + \cos (\angle\, \tilde{\mu}_{2}^{(t_{0})}(ct_{0})).
\end{equation}
From (\ref{lem4.4-1.8new}) and (\ref{lem4.4-1.9new}), 
we hence get (\ref{lem4.4-1}). 
As well as above, since $\varphi$ is differentiable at $t = t_{0}$, we also get 
\begin{equation}\label{lem4.4-2}
\varphi'(t_{0}) 
= 
\cos (\angle\,\mu_{1}(at_{0})) + \cos (\angle\, \mu_{2}(ct_{0})).
\end{equation}
By (\ref{lem4.4-angle}), (\ref{lem4.4-1}), and (\ref{lem4.4-2}), we get 
$\varphi'(t_{0}) \le \wt{\psi}'(t_{0})$. 
Hence, we conclude that $\cD'(t_{0}) \le 0$ (see the proof of  \cite[Lemma 4.4]{KT2}). 
Thus, $\cD'(t) \le 0$ for almost all $t \in (0, 1)$. 
This implies that $\cD(t)$ is non-increasing, since $\cD (t)$ is locally Lipschitz.
$\qedd$
\end{proof}

\medskip

\begin{remark}
As pointed out in \cite[Remark 4.5]{KT2}, it is a very important property that $\cD (t)$ is locally Lipschitz. 
Without this property, we can not conclude that $\cD(t)$ is non-increasing.
\end{remark}

\section{Toponogov's comparison theorem}\label{sec:TCTproof}
Our purpose of this section is to prove our main theorem, i.e., 
the Toponogov comparison theorem for open triangles (Theorem \ref{thm4.9}), 
by using new techniques established in \cite[Section 4]{KT2} and 
Lemmas \ref{lem4.1}, \ref{lem3.8}, and \ref{lem4.4}. 

\bigskip

Throughout this section, let $(\wt{X}, \partial \wt{X})$ denote 
a model surface with its metric (\ref{model-metric}). 

\begin{lemma}\label{lem4.6}
Let 
\[
{\rm OT}(\partial \wt{X}, \tilde{p}_{1}, \tilde{q}_{1}) 
= 
(\partial \wt{X}, \tilde{p}_{1}, \tilde{q}_{1})
\,;\;\tilde{\gamma}_{1}, \tilde{\mu}^{(1)}_{1}, \tilde{\mu}^{(1)}_{2})
\] and 
\[
{\rm OT}(\partial \wt{X}, \tilde{p}_{2}, \tilde{q}_{2}) 
= 
(\partial \wt{X}, \tilde{p}_{2}, \tilde{q}_{2}\,;\,
\tilde{\gamma}_{2}, \tilde{\mu}^{(2)}_{1}, \tilde{\mu}^{(2)}_{2})
\]
be open triangles in $\wt{X}$ such that 
\begin{equation}\label{lem4.6-length1}
d(\partial \wt{X}, \tilde{q}_{1}) = d(\partial \wt{X}, \tilde{p}_{2}),
\end{equation}
and that 
\begin{equation}\label{lem4.6-angle1}
\angle\,\tilde{q}_{1} + \angle\,\tilde{p}_{2} \le \pi.
\end{equation}
If there exists an open triangle 
${\rm OT}(\partial \wt{X}, \tilde{p}, \tilde{q}) 
= 
(\partial \wt{X}, \tilde{p}, \tilde{q}\,;\,\tilde{\gamma}, \tilde{\mu}_{1}, \tilde{\mu}_{2})
$ 
in a sector $\wt{X}(\theta_{0})$, which has no pair of cut points, satisfying 
\begin{equation}\label{lem4.6-length2}
d(\partial \wt{X}, \tilde{p}) = d(\partial \wt{X}, \tilde{p}_{1}), \quad 
d(\tilde{p}, \tilde{q}) = d(\tilde{p}_{1}, \tilde{q}_{1}) + d(\tilde{p}_{2}, \tilde{q}_{2}), \quad 
d(\partial \wt{X}, \tilde{q}) = d(\partial \wt{X}, \tilde{q}_{2}),
\end{equation}
then
\[
\angle\,\tilde{p}_{1} \ge \angle\,\tilde{p}, \quad \angle\,\tilde{q}_{2} \ge \angle\,\tilde{q}.
\]
\end{lemma}

\begin{proof}
By (\ref{lem4.6-length1}), we may assume that ${\rm OT}(\partial \wt{X}, \tilde{p}_{2}, \tilde{q}_{2})$ 
is adjacent to ${\rm OT}(\partial \wt{X}, \tilde{p}_{1}, \tilde{q}_{1})$ so as to have a common side 
$\tilde{\mu}^{(1)}_{2} = \tilde{\mu}^{(2)}_{1}$, i.e., $\tilde{q}_{1} = \tilde{p}_{2}$. 
We may also assume that 
\[
0 = \tilde{y}(\tilde{p}_{1}) < \tilde{y}(\tilde{q}_{1}) = \tilde{y}(\tilde{p}_{2}) < \tilde{y}(\tilde{q}_{2}).
\]
Furthermore, we may assume that $\tilde{p} = \tilde{p}_{1}$ and $\tilde{y} (\tilde{q}) > 0$.
Remark that $\tilde{\mu}_{1} = \tilde{\mu}^{(1)}_{1}$. 
If $\angle\,\tilde{q}_{1} + \angle\,\tilde{p}_{2} = \pi$ holds, 
then there is nothing to prove. 
Thus, by (\ref{lem4.6-angle1}), we may assume that 
$\angle\,\tilde{q}_{1} + \angle\,\tilde{p}_{2} < \pi$. 
Hence, from the triangle inequality and (\ref{lem4.6-length2}), 
we see 
\begin{equation}\label{lem4.6-1}
d(\tilde{p}_{1}, \tilde{q}_{2}) < d(\tilde{p}, \tilde{q}) =  d(\tilde{p}_{1}, \tilde{q}).
\end{equation}
Since $\tilde{q} \in \wt{X} (\theta_{0})$, 
it follows from Lemma \ref{lem4.1} and (\ref{lem4.6-1}) that 
\[
\tilde{y}(\tilde{q}_{2}) < \tilde{y}(\tilde{q}) < \theta_{0}.
\]
Thus, $\tilde{\gamma}_{2}$ lies in $\wt{X} (\theta_{0})$. 
Since $\wt{X}(\theta_{0})$ has no pair of cut points, 
the geodesic extension $\tilde{\sigma}$ of $\tilde{\gamma}_{1}$ does not 
intersect the side $\tilde{\gamma}_{2}$ except for $\tilde{q}_{1}$. 
We will prove that $\tilde{\sigma}$ does not intersect 
$\tilde{\tau} ([\tilde{y}(\tilde{q}_{2}), \tilde{y}(\tilde{q})])$, 
where $\tilde{\tau}$ 
denotes $\tilde{\tau} (t) := (\tilde{x} (\tilde{q}_{2}), t) \in \wt{X}$. 
Suppose that $\tilde{\sigma}$ intersect $\tilde{\tau} ([\tilde{y}(\tilde{q}_{2}), \tilde{y}(\tilde{q})])$ 
at a point $\tilde{\sigma} (s_{0})$. 
From Lemma \ref{lem4.1}, 
we have 
\begin{equation}\label{lem4.6-2}
d(\tilde{q}_{1}, \tilde{q}_{2}) < d(\tilde{q}_{1}, \tilde{\sigma} (s_{0})).
\end{equation}
Notice that $\tilde{\sigma}(s_{0}) \not= \tilde{q}_{2}$, 
since $\tilde{\gamma}_{2}$ does not meet $\tilde{\sigma}$ except for $\tilde{q}_{1}$. 
Thus, by (\ref{lem4.6-length2}) and (\ref{lem4.6-2}), 
\[
d(\tilde{p}_{1}, \tilde{q}) 
< d(\tilde{p}_{1}, \tilde{q}_{1}) + d(\tilde{q}_{1}, \tilde{\sigma} (s_{0})) 
= d(\tilde{p}_{1}, \tilde{\sigma} (s_{0})).
\]
Hence, by applying Lemma \ref{lem4.1} again, we get 
$\tilde{y}(\tilde{q}) < \tilde{y}(\tilde{\sigma} (s_{0}))$. 
This is impossible, 
since 
$\tilde{\sigma}(s_{0}) \in \tilde{\tau} ((\tilde{y}(\tilde{q}_{2}), \tilde{y}(\tilde{q})])$. 
Therefore, we have proved that $\tilde{\sigma}$ does not intersect 
$\tilde{\tau} ([\tilde{y}(\tilde{q}_{2}), \tilde{y}(\tilde{q})])$.\par
If the extension $\tilde{\sigma}$ intersects $\partial \wt{X}$ at a point $\tilde{\sigma} (s_{1})$ 
in $\wt{X} (\theta_{0})$, 
then we denote by $\wt{\cA} (\theta_{0})$ the domain bounded 
by $\tilde{\mu}_{1}$ and $\tilde{\sigma} ([0, s_{1}])$. 
If $\tilde{\sigma}$ does not intersect $\partial \wt{X}$ in $\wt{X} (\theta_{0})$, 
then $\tilde{\sigma}$ intersects the $\partial \wt{X}$-ray $\tilde{y} = \theta_{0}$ 
at a point $\tilde{\sigma} (s_{2})$. 
In this case, $\wt{\cA} (\theta_{0})$ denotes the domain bounded by $\tilde{\mu}_{1}$, 
$\tilde{\sigma} ([0, s_{2}])$, and the $\partial \wt{X}$-segment to $\tilde{\sigma} (s_{2})$. 
By the argument above, 
the point $\tilde{q}$ lies in the domain $\wt{\cA} (\theta_{0})$. 
Hence, the opposite side $\tilde{\gamma}$ of ${\rm OT}(\partial \wt{X}, \tilde{p}, \tilde{q})$ to 
$\partial \wt{X}$ must lie in the closure of $\wt{\cA} (\theta_{0})$, since $\wt{X} (\theta_{0})$ has no pair of cut points. 
In particular, $\angle\,\tilde{p}_{1} \ge \angle\,\tilde{p}$ is now clear. 
By repeating the same argument above for the pair of open triangles 
${\rm OT}(\partial \wt{X}, \tilde{q}_{2}, \tilde{p}_{2})$ and 
${\rm OT}(\partial \wt{X}, \tilde{q}_{1}, \tilde{p}_{1})$, 
we also get $\angle\,\tilde{q}_{2} \ge \angle\,\tilde{q}$. $\qedd$
\end{proof}

\bigskip

From now on, we denote by $(X, \partial X)$ 
a complete connected Riemannian $n$-dimensional manifold $X$ with smooth 
{\bf convex} boundary $\partial X$ whose radial curvature is bounded from below by that 
of $(\wt{X}, \partial \wt{X})$. 

\begin{lemma}\label{lem4.7}
If an open triangle ${\rm OT}(\partial X, p, q) = (\partial X, p, q\,;\,\gamma, \mu_{1}, \mu_{2})$ in $X$ 
admits an open triangle 
${\rm OT}(\partial \wt{X}, \tilde{p}, \tilde{q}) 
= 
(\partial \wt{X}, \tilde{p}, \tilde{q}\,;\,\tilde{\gamma}, \tilde{\mu}_{1}, \tilde{\mu}_{2})
$ 
in a sector $\wt{X}(\theta_{0})$ satisfying 
\begin{equation}\label{lem4.7-0}
d(\partial \wt{X}, \tilde{p}) = d(\partial X, p), \quad 
d(\tilde{p}, \tilde{q}) = d(p, q), \quad 
d(\partial \wt{X}, \tilde{q}) = d(\partial X, q),
\end{equation}
then, for any $s \in (0, d(p, q))$, there exists an open triangle 
${\rm OT}(\partial \wt{X}, \tilde{p}, \tilde{\gamma}(s))$ in $\wt{X}(\theta_{0})$ 
satisfying $(\ref{lem4.7-0})$ for $q = \gamma (s)$ and $\tilde{q} = \tilde{\gamma}(s)$. 
\end{lemma}

\begin{proof}
It is clear from Lemmas \ref{lem3.8} and \ref{lem4.6}. 
See also the proof of \cite[Lemma 4.9]{KT2}.
$\qedd$
\end{proof}

\begin{proposition}\label{lem4.8}
Let ${\rm OT}(\partial X, p, q) = (\partial X, p, q\,;\,\gamma, \mu_{1}, \mu_{2})$ 
be an open triangle in $X$. 
Then, there exists an open triangle 
${\rm OT}(\partial \wt{X}, \tilde{p}, \tilde{q}) = 
(\partial \wt{X}, \tilde{p}, \tilde{q}\,;\,\wt{\gamma}, \wt{\mu}_{1}, \wt{\mu}_{2})$ 
in $\wt{X}$ satisfying 
\begin{equation}\label{lem4.8-length}
d(\partial \wt{X}, \tilde{p}) = d(\partial X, p), \quad 
d(\tilde{p}, \tilde{q}) = d(p, q), \quad 
d(\partial \wt{X}, \tilde{q}) = d(\partial X, q).
\end{equation}
Furthermore, if the ${\rm OT}(\partial \wt{X}, \tilde{p}, \tilde{q})$ lies in a sector $\wt{X}(\theta_{0})$, 
which has no pair of cut points, then 
\begin{equation}\label{lem4.8-angle}
\angle\,p \ge \angle\,\tilde{p}, \quad \angle\,q \ge \angle\,\tilde{q}.
\end{equation}
\end{proposition}

\begin{proof}
Since $\mu_{1}$ (resp. $\mu_{2}$) is the $\partial X$-segment to $p$ (resp. to $q$), 
we obtain 
\[
c \le a + b, \quad a \le b + c.
\]
Here we set 
$a := d(\partial X, p)$, $b := d(p, q)$, and $c := d(\partial X, q)$. 
Hence, we have $|a -c| \le b$. 
Choose any point $\tilde{p} \in \wt{X}$ satisfying $d(\partial \wt{X}, \tilde{p}) = a$, and fix the point. 
Since we have $d (\tilde{p}, \tilde{\tau}_{c}(s)) = |a - c|$ at $s = \tilde{y} (\tilde{p})$ and 
$\lim_{s \to \infty} d (\tilde{p}, \tilde{\tau}_{c}(s)) = \infty$,
we may find a number $s_{0} \ge \tilde{y} (\tilde{p})$ such that 
$d (\tilde{p}, \tilde{\tau}_{c}(s_{0})) = b$. 
Here $\tilde{\tau}_{c}$ denotes the arc $\tilde{x} = c$, i.e., $\tilde{\tau}_{c} (s) = (c, s) \in \wt{X}$. 
Putting $\tilde{q} := \tilde{\tau}_{c}(s_{0})$, 
we therefore find a triangle ${\rm OT}(\partial \wt{X}, \tilde{p}, \tilde{q})$ 
satisfying (\ref{lem4.8-length}).\par
Hereafter, we assume that the 
${\rm OT}(\partial \wt{X}, \tilde{p}, \tilde{q})$ lies in the sector $\wt{X}(\theta_{0})$. 
Let $S$ be the set of all $s \in (0, d(p, q))$ such that 
there exists an open triangle 
${\rm OT}(\partial \wt{X}, \tilde{p}, \tilde{\gamma} (s)) \subset \wt{X}(\theta_{0})$ 
corresponding to the triangle ${\rm OT}(\partial X, p, \gamma (s)) \subset X$ satisfying 
(\ref{lem4.8-length}) and (\ref{lem4.8-angle}) for $q = \gamma(s)$ and $\tilde{q} = \tilde{\gamma}(s)$. 
Since ${\rm OT}(\partial X, p, \gamma (\ve)) \subset X$ is a thin open triangle in $X$ 
for any sufficiently small $\ve > 0$, it follows from Lemma \ref{lem3.8} that $S$ is non-empty. 
Since there is nothing to prove in the case where $\sup S = d(p, q)$, 
we then suppose that 
\[
s_{1} := \sup S < d(p, q).
\]
Since $s_{1} \in S$, there exists an open triangle 
${\rm OT}(\partial \wt{X}, \tilde{p}_{1}, \tilde{q}_{1}) \subset \wt{X} (\theta_{0})$ 
corresponding to the triangle ${\rm OT}(\partial X, p, \gamma (s_{1})) \subset X$ such that 
(\ref{lem4.8-length}) and (\ref{lem4.8-angle}) hold 
for $q = \gamma(s_{1})$, $\tilde{p} = \tilde{p}_{1}$, and $\tilde{q} = \tilde{q}_{1}$. 
Choose any $\ve_{1} \in (0, d(p, q) - s_{1})$ in such a way that the open triangle 
${\rm OT}(\partial X, \gamma (s_{1}), \gamma (s_{1} + \ve_{1})) \subset X$ is thin. 
From Lemma \ref{lem3.8}, there exists an open triangle 
${\rm OT}(\partial \wt{X}, \tilde{p}_{2}, \tilde{q}_{2}) \subset \wt{X}$ 
corresponding to the 
${\rm OT}(\partial X, \gamma (s_{1}), \gamma (s_{1} + \ve_{1})) \subset X$ 
such that (\ref{lem4.8-length}) and (\ref{lem4.8-angle}) hold 
for $p = \gamma(s_{1})$, $q = \gamma(s_{1} + \ve_{1})$, $\tilde{p} = \tilde{p}_{2}$, 
and $\tilde{q} = \tilde{q}_{2}$. 
It is clear that the pair of open triangles ${\rm OT}(\partial \wt{X}, \tilde{p}_{1}, \tilde{q}_{1})$ and 
${\rm OT}(\partial \wt{X}, \tilde{p}_{2}, \tilde{q}_{2})$ satisfy 
(\ref{lem4.6-length1}) and (\ref{lem4.6-angle1}) in Lemma \ref{lem4.6}. 
For this pair, it is clear from Lemma \ref{lem4.7} that there exists an open triangle 
${\rm OT}(\partial \wt{X}, \wh{p}, \wh{q}\,) \subset \wt{X} (\theta_{0})$ such that (\ref{lem4.6-length2}) 
holds for $\tilde{p} = \wh{p}$ and $\tilde{q} = \wh{q}$. 
This implies that $s_{1} + \ve_{1} \in S$. 
This therefore contradicts the fact that $s_{1} = \sup S$.
$\qedd$
\end{proof}

\begin{theorem}\label{thm4.9}{\bf (Toponogov's comparison theorem for open triangles)} 
Let $(X, \partial X)$ be a complete connected Riemannian 
$n$-dimensional manifold $X$ with smooth convex boundary $\partial X$ 
whose radial curvature is bounded from below by that of 
a model surface $(\wt{X}, \partial \wt{X})$ with its metric $(\ref{model-metric})$. 
Assume that 
$\wt{X}$ admits a sector $\wt{X}(\theta_{0})$ which has no pair of cut points. 
Then, for every open triangle ${\rm OT}(\partial X, p, q) = (\partial X, p, q\,;\,\gamma, \mu_{1}, \mu_{2})$ in $X$ 
with 
\[
d (\mu_{1}(0), \mu_{2}(0)) < \theta_{0},
\]
there exists an open triangle 
${\rm OT}(\partial \wt{X}, \tilde{p}, \tilde{q}) = 
(\partial \wt{X}, \tilde{p}, \tilde{q}\,;\,\tilde{\gamma}, \tilde{\mu}_{1}, \tilde{\mu}_{2})$ in $\wt{X}(\theta_{0})$ 
such that
\begin{equation}\label{thm4.9-length}
d(\partial \wt{X},\tilde{p}) = d(\partial X, p), \quad 
d(\tilde{p},\tilde{q}) = d(p, q), \quad 
d(\partial \wt{X},\tilde{q}) = d(\partial X, q)
\end{equation}
and that
\begin{equation}\label{thm4.9-angle}
\angle\,p \ge \angle\,\tilde{p}, \quad  
\angle\,q \ge \angle\,\tilde{q}, \quad 
d (\mu_{1}(0), \mu_{2}(0)) \ge d (\tilde{\mu}_{1}(0), \tilde{\mu}_{2}(0)).
\end{equation}
Furthermore, if 
\[
d (\mu_{1}(0), \mu_{2}(0)) = d (\tilde{\mu}_{1}(0), \tilde{\mu}_{2}(0))
\] 
holds, then 
\[
\angle\,p = \angle\,\tilde{p}, \quad  \angle\,q = \angle\,\tilde{q}
\]
hold. 
\end{theorem}

\begin{proof}
Since the claim of our theorem is trivial for degenerate open triangles, 
we assume that the open triangle ${\rm OT}(\partial X, p, q)$ is not degenerate. 
Here, we make use of the same notations used in Lemma \ref{lem4.4} and its proof.\par
Applying the triangle inequality to the open triangle 
${\rm OT}(\partial X, \mu_{1}(at), \mu_{2}(ct)) \subset X$, we see
\begin{equation}\label{thm4.9-1}
\varphi (t) - (a + c)t \le d(\mu_{1}(0), \mu_{2}(0)) \le \varphi (t) + (a + c)t
\end{equation}
for all $t \in (0, 1]$, where $a := d(\partial X, p)$, $c := d(\partial X, q)$, and 
$\varphi (t) := d(\mu_{1}(at), \mu_{2}(ct))$. 
By the first assertion of Proposition \ref{lem4.8}, for each $t \in (0, 1]$, 
we may find an open triangle 
\[
{\rm OT}(\partial \wt{X}, \tilde{\mu}_{1}^{(t)}(at), \tilde{\mu}_{2}^{(t)}(ct)) 
= 
(\partial \wt{X}, \tilde{\mu}_{1}^{(t)}(at), \tilde{\mu}_{2}^{(t)}(ct)\,;\,
\tilde{\gamma}_{t}, \tilde{\mu}_{1}^{(t)}|_{[0, \,at]}, \tilde{\mu}_{2}^{(t)}|_{[0, \,ct]})
\]
in $\wt{X}$ which has the same side lengths as the ${\rm OT}(\partial X, \mu_{1}(at), \mu_{2}(ct))$. 
Thus, as well as (\ref{thm4.9-1}), we see 
\begin{equation}\label{thm4.9-2}
\varphi (t) - (a + c)t \le d(\tilde{\mu}_{1}^{(t)}(0), \tilde{\mu}_{2}^{(t)}(0)) \le \varphi (t) + (a + c)t
\end{equation}
for all $t \in (0, 1]$. 
From (\ref{thm4.9-1}) and (\ref{thm4.9-2}), we obtain 
\begin{equation}\label{thm4.9-3}
d(\mu_{1}(0), \mu_{2}(0)) - 2(a + c)t \le 
d(\tilde{\mu}_{1}^{(t)}(0), \tilde{\mu}_{2}^{(t)}(0)) \le 
d(\mu_{1}(0), \mu_{2}(0)) + 2(a + c)t
\end{equation}
for all $t \in (0, 1]$. 
Since $d(\mu_{1}(0), \mu_{2}(0)) < \theta_{0}$, it follows from (\ref{thm4.9-3}) that 
there exists a number $\ve_{1} > 0$ such that 
\begin{equation}\label{thm4.9-4}
d(\tilde{\mu}_{1}^{(t)}(0), \tilde{\mu}_{2}^{(t)}(0)) < \theta_{0}
\end{equation}
holds on $(0, \ve_{1})$. 
Hence, 
\begin{equation}\label{thm4.9-5}
{\rm OT}(\partial \wt{X}, \tilde{\mu}_{1}^{(t)}(at), \tilde{\mu}_{2}^{(t)}(ct)) \subset \wt{X} (\theta_{0})
\end{equation}
for each $t \in (0, \ve_{1})$. 
By the second assertion of Proposition \ref{lem4.8}, we get 
\begin{equation}\label{thm4.9-6}
\angle\,\mu_{1}(at) \ge \angle\,\tilde{\mu}_{1}^{(t)} (at), \quad 
\angle\,\mu_{2}(ct) \ge \angle\,\tilde{\mu}_{2}^{(t)} (ct)
\end{equation}
for each $t \in (0, \ve_{1})$. 
Since $\wt{X}(\theta_{0})$ has no pair of cut points, 
it follows from (\ref{thm4.9-5}) that the opposite side $\tilde{\gamma}_{t}$ of 
${\rm OT}(\partial \wt{X}, \tilde{\mu}_{1}^{(t)}(at), \tilde{\mu}_{2}^{(t)}(ct))$ 
to $\partial \wt{X}$ is unique for all $t \in (0, \ve_{1})$. 
From Lemma \ref{lem4.4} and (\ref{thm4.9-4}), 
it follows that the function $\cD (t) = d(\tilde{\mu}_{1}^{(t)}(0), \tilde{\mu}_{2}^{(t)}(0))$ is non-increasing 
on $(0, \ve_{1})$ and $\cD (t) < \theta_{0}$ holds on $(0, \ve_{1})$. 
Thus, we finally see that $\cD (t)$ is non-increasing on $(0, 1]$, $\cD (t) < \theta_{0}$ holds on $(0, 1]$, 
and (\ref{thm4.9-6}) holds on $(0, 1]$. 
In particular, setting 
\[
{\rm OT}(\partial \wt{X}, \tilde{p}, \tilde{q}) 
= (\partial \wt{X}, \tilde{p}, \tilde{q}\,;\,\tilde{\gamma}, \tilde{\mu}_{1}, \tilde{\mu}_{2}) 
:= {\rm OT}(\partial \wt{X}, \tilde{\mu}_{1}^{(1)}(a), \tilde{\mu}_{2}^{(1)}(c)) 
\]
in $\wt{X} (\theta_{0})$, we get 
\begin{equation}\label{thm4.9-7}
\angle\,p \ge \angle\,\tilde{p}, \quad  
\angle\,q \ge \angle\,\tilde{q}.
\end{equation}
Moreover, by (\ref{thm4.9-3}), 
\begin{equation}\label{thm4.9-8}
\cD(t) = d(\tilde{\mu}_{1}^{(t)}(0), \tilde{\mu}_{2}^{(t)}(0)) \le d(\mu_{1}(0), \mu_{2}(0)) + 2(a + c)t
\end{equation}
holds on $(0, 1]$. 
Since $\cD (t)$ is non-increasing on $(0, 1]$, we have, by (\ref{thm4.9-8}), 
\[
\cD(1) = 
d(\tilde{\mu}_{1}(0), \tilde{\mu}_{2}(0)) 
\le d(\mu_{1}(0), \mu_{2}(0)) + 2(a + c)t
\]
on $(0, 1]$. 
Hence we get 
\begin{equation}\label{thm4.9-9}
d(\mu_{1}(0), \mu_{2}(0)) \ge d(\tilde{\mu}_{1}(0), \tilde{\mu}_{2}(0)). 
\end{equation}
By (\ref{thm4.9-7}) and (\ref{thm4.9-9}), 
the open triangle 
${\rm OT}(\partial \wt{X}, \tilde{p}, \tilde{q})$ is therefore an open triangle 
satisfying conditions (\ref{thm4.9-length}) and (\ref{thm4.9-angle}).\par 
Assume that 
$d(\mu_{1}(0), \mu_{2}(0)) = d(\tilde{\mu}_{1}(0), \tilde{\mu}_{2}(0)) = \cD(1)$ holds. 
By (\ref{thm4.9-8}), 
\[
\cD(t) \le 2(a + c)t + \cD(1)
\]
holds on $(0, 1]$. 
Thus, we get 
\[
\lim_{t \downarrow 0} \cD(t) \le \cD(1).
\]
Hence, $\cD(t)$ must be constant on $(0, 1]$, 
since $\cD (t)$ is non-increasing on $(0, 1]$. 
From the proof of Lemma \ref{lem4.4}, it follows that 
$\angle\,\mu_{1}(at) = \angle\,\tilde{\mu}_{1}^{(t)} (at)$ 
and 
$\angle\,\mu_{2}(ct) = \angle\,\tilde{\mu}_{2}^{(t)} (ct)$ 
hold on $(0, 1]$. 
In particular, we obtain $\angle\,p = \angle\,\tilde{p}$ and $\angle\,q = \angle\,\tilde{q}$.
$\qedd$
\end{proof}

\medskip

\begin{center}
Department of Mathematics\\
Tokai University\\
Hiratsuka City, Kanagawa Pref.\\ 
259\,--\,1292\\
Japan

\bigskip

{\small
$\bullet$\,our e-mail addresses\,$\bullet$

\bigskip 
\textit{e-mail of Kondo} \,:

\medskip
{\tt keikondo@keyaki.cc.u-tokai.ac.jp}

\medskip
\textit{e-mail of Tanaka}\,:

\medskip
{\tt tanaka@tokai-u.jp}
}
\end{center}


\begin{thebibliography}{MMMM}

\bibitem[CG]{CG}
J.~Cheeger and D.~Gromoll, 
{\it On the structure of complete manifolds of nonnegative curvature}, 
Ann.\ of \ Math. (2)
\textbf{96} (1972), 413--443.

\bibitem[CV1]{CV1}
S.~Cohn\,-Vossen, 
{\it K\"urzeste Wege und Totalkr\"ummung auf Fl\"achen}, 
Compositio Math. \textbf{2} (1935), 63--113.

\bibitem[CV2]{CV2}
S.~Cohn\,-Vossen, 
{\it Totalkr\"ummung und geod\"atische Linien auf einfach zusammenh\"angenden 
offenen volst\"andigen Fl\"achenst\"ucken}, 
Recueil Math. Moscow 
\textbf{43} (1936), 139--163.

\bibitem[D]{D}
U.~Dini, {\it Fondamenti per la teorica delle funzioni di variabili reali},
Pisa, 1878.

\bibitem[G]{G}
D.~Gromoll, {\it The e-mail to authors},
the 11th October, 2007.

\bibitem[GM]{GM}
D.~Gromoll and W.~Meyer, {\it On complete manifolds of positive curvature},
Ann.\ of Math.\ $(2)$ \textbf{75} (1969), 75--90.

\bibitem[H]{H}
T.~Hawkins, Lebesgue's theory of integration\,: Its origins and development, 
University of Wisconsin Press, Madison, 1970.

\bibitem[IT1]{IT1}
J.~Itoh and M.~Tanaka, 
{\it The dimension of a cut locus on a smooth Riemannian manifold}, 
Tohoku Math. J. \textbf{50} (1998), 571--575.

\bibitem[IT2]{IT3}
J.~Itoh and M.~Tanaka, 
{\it The Lipschitz continuity of the
distance function to the cut locus}, 
Trans. Amer. Math. Soc. \textbf{353} (2001), 21--40.

\bibitem[IMS]{IMS}
Y.~Itokawa, Y.~Machigashira and K.~Shiohama,
{\it Generalized Toponogov's theorem for manifolds 
with radial curvature bounded below},
Explorations in complex and Riemannian geometry, 121--130,
Contemp.\ Math. \textbf{332}, Amer.\ Math.\ Soc., Providence, RI, 2003.

\bibitem[KT1]{KT1}
K.~Kondo and M.~Tanaka, 
{\it Total curvatures of model surfaces control 
topology of complete open manifolds with radial curvature bounded below.\,I}, 
to appear in Math. Ann., DOI: 10.1007/s00208-010-0593-4. 

\bibitem[KT2]{KT2}
K.~Kondo and M.~Tanaka, 
{\it Total curvatures of model surfaces control 
topology of complete open manifolds with radial curvature bounded below.\,II}, 
Trans. Amer. Math. Soc. \textbf{362} (2010), 6293--6324.


\bibitem[KT3]{KT4}
K.~Kondo and M.~Tanaka, 
{\it Applications of the Toponogov comparison theorem for open triangles}, 
 {\tt arXiv:1102.4156}

\bibitem[MS]{MS}
Y.~Mashiko and K.~Shiohama, 
{\it Comparison geometry referred to the warped product models}, 
Tohoku Math. J. \textbf{58} (2006), 461--473.

\bibitem[S]{S}
T.~Sakai, Riemannian geometry, Transl. Math. Monogr., 
\textbf{149}., American Mathematical Society, Providence, R.I., 1996.

\bibitem[SST]{SST}
K.~Shiohama, T.~Shioya and M.~Tanaka, 
The geometry of total curvature on complete open surfaces,
Cambridge Tracks in Math.  \textbf{159}, Cambridge University Press, Cambridge, 2003.

\bibitem[Tm]{Tm}
K.~Tamura, {\it On the cut loci of a complete Riemannian manifold homeomorphic to a cylinder},
Master thesis, Tokai University, 2003.

\bibitem[Tn]{Tn}
M.~Tanaka, {\it On the cut loci of a von Mangoldt's surface of revolution},
J.\ Math.\ Soc.\ Japan \textbf{44} (1992), 631--641.

\bibitem[TK]{KT3}
M.~Tanaka and K.~Kondo, 
{\it The Gauss curvature of a model surface with 
finite total curvature is not always bounded.}, {\tt arxiv:1102.0852}

\bibitem[T1]{T1}
V.\,A.~Toponogov, {\it Riemann spaces with curvature bounded below} (in Russian), 
Uspehi Mat. Nauk \textbf{14} (1959), 87--130.

\bibitem[T2]{T2}
V.\,A.~Toponogov, {\it Riemannian spaces containing straight lines} (in Russian), 
Dokl.\ Akad.\ Nauk \ SSSR \textbf{127} (1959), 977--979.

\bibitem[W]{W}
F.\,W.\,Warner, 
{\it Extension of the Rauch comparison theorem to submanifolds}, 
Trans. Amer.\ Math.\ Soc. \textbf{122} (1966), 341--356.

\bibitem[WZ]{WZ}
R.\ L.~Wheeden and A.~Zygmund, Measure and integral, An introduction to real analysis, 
Pure and Applied Mathematics, vol. 34, 
Marcel Decker Inc., New York, 1977.

\end{thebibliography}
\end{document}